\documentclass{amsart}
 \newtheorem{theorem}{Theorem}[section]
 \newtheorem{corollary}[theorem]{Corollary}
 \newtheorem{lemma}[theorem]{Lemma}
 \newtheorem{proposition}[theorem]{Proposition}
 \theoremstyle{definition}
 \newtheorem{definition}[theorem]{Definition}
 \theoremstyle{remark}
 \newtheorem{remark}[theorem]{Remark}
 
 \numberwithin{equation}{section}

\def \bC {\mathbb C}

\def \bN {\mathbb N}

\def \bR {\mathbb R}

\def \bT {\mathbb T}

\def \bZ {\mathbb Z}

\def \cB {\mathcal B}

\def \cD {\mathcal D}

\def \cF {\mathcal F}

\def \cH {\mathcal H}

\def \cK {\mathcal K}
\def \cL {\mathcal L}
\def \cM {\mathcal M}

\def \cR {\mathcal R}
\def \cS {\mathcal S}
\def \cR {\mathcal R}
\def \cR {\mathcal R}

\def \fg {\mathfrak g}

\def \fU {\mathfrak U}

\def \tr {\text{\rm Tr}}
\def \pr {\text{\rm proj}}

\def\Rep{\text{\rm Re}\ }
\def\dom{\text{\rm Dom}\ }
\def\id{\text{\rm Id} }

\def\Gh{\widehat{G}}
\def\wh{\widehat}
\def\whfpi{{\wh f(\pi)}}

\usepackage{mathrsfs} 

\def \sL{\mathscr L}

\begin{document}
%
%
%
%
%
%
%
%
%
\title[A pseudo-differential calculus on graded nilpotent Lie groups]
 {A pseudo-differential calculus\\ on graded nilpotent Lie groups}
\author[V. Fischer]{V\'eronique Fischer}

\address{Universita degli studi di Padova, DMMMSA, Via Trieste 63,
           35121 Padova, Italy}

\email{fischer@dmsa.unipd.it}

\thanks{The first author acknowledges the support of the London Mathematical Society via the Grace Chisholm Fellowship. 
It was during this fellowship held at King's College London in 2011 that the work was initiated.
The second author was supported in part by the EPSRC Leadership Fellowship  EP/G007233/1.}
\author[M. Ruzhansky]{Michael Ruzhansky}
\address{180 Queen's Gate, Department of Mathematics, Imperial College London, London, 
SW7 2AZ, United Kingdom}
\email{m.ruzhansky@imperial.ac.uk}

\subjclass{Primary 35S05; Secondary 43A80}

\keywords{Pseudo-differential operators, nilpotent Lie groups}

\date{July 2012, revised May 2013}

\begin{abstract}
In this paper, 
we present first results of our investigation 
regarding symbolic pseudo-differential calculi
on nilpotent Lie groups.
On any graded Lie group, 
we define classes of symbols using difference operators.
The operators are obtained
from these symbols
via the natural quantisation  given by the representation theory.
They form an algebra of operators 
which shares many properties with the usual H\"ormander calculus.
\end{abstract}

\maketitle


\section{Introduction}

In the last five decades 
pseudo-differential operators 
have become a standard tool in the study of partial differential equations.
It is natural to try 
to define analogues of the Euclidean pseudo-differential calculus in other settings.
On one hand, while it is always possible to obtain a local calculus on any (connected) manifold, 
the question becomes much harder for global calculi.
If, in addition, one requires a notion of symbol,
the quasi inherent context is the one of Lie groups of type~I
where a Plancherel-Fourier analysis is available. 
On the other hand,
from the viewpoint of  
what can actually be done at the level of operators,
the investigation should start 
in the context of Lie groups with polynomial volume-growth
where analysis of integral operators  is quite well understood.
Therefore the natural setting to start developing global pseudo-differential calculi
is nilpotent or compact Lie groups, together with their semi-direct products.

The genesis of this idea began quite some time ago;
if a starting line had to be drawn, 
it would be in the seventies 
with the work of Elias Stein 
and his collaborators
Folland, Rotschild, etc. 
(see e.g. \cite{folland+stein-1974,rothschild+stein}), 
and
continued, in the next decade,  
with the work of Beals and Greiner amongst many others.
Their motivation came from the study 
of differential operators on CR or contact manifolds,
modelling locally the operators 
on homogeneous left-invariant convolution operators on nilpotent groups
(cf \cite{ponge}). 
In `practice' and from this motivation, 
only nilpotent Lie groups 
 endowed with some compatible structure of dilations, 
 i.e. homogeneous  groups and, more particularly, graded Lie groups,
 are considered. 
 The latter is also the setting of our present investigation.

Since the seventies, 
several global calculi of operators on homogeneous Lie groups
have appeared.
However 
they were often calculi of left-invariant operators
with the following notable exceptions 
to the authors' knowledge.
Beside Dynin's construction of certain operators 
on the Heisenberg group in  \cite{dynin,folland-1994},
a non-invariant pseudo-differential calculus on any homogeneous group
 was developed in \cite{cggp}
 but this is not symbolic since the operator classes are defined via properties of the kernel.
In the revised version of \cite{Tnma},
Taylor  describes several (non-invariant) operator calculi
and, in a different direction, he also explains a way to develop symbolic calculi:
using the representations of the group, he defines 
a general \emph{quantization} 
and the natural \emph{symbols} on any unimodular type I group
(by quantization,  we mean a procedure which associates an operator to a symbol).
He illustrates this on the Heisenberg group 
and obtains there several important applications for, e.g., the study of hypoellipticity.
He uses the fact that, 
 because of the properties of the Schr\"odinger representations of the Heisenberg group, a symbol is a family of operators in the Euclidean space,
themselves given by symbols via the Weyl quantization. 
Recently, 
the attempt at defining suitable classes of Shubin type for these Weyl-symbols  led to another
version of the calculus on the Heisenberg group in \cite{bahouri+FK+gallagher_bk2012}\footnote{and in its revised version}.

Recently as well,
using the global quantization procedure noted in \cite{Tnma},
the second author and Turunen
developed a global symbolic calculus on any compact Lie group  
in \cite{ruzh+turunen_bk2010}.
They successfully defined symbol classes 
so that the quantization procedure makes sense 
and the resulting operators form an algebra of operators 
with properties `close enough' to the one enjoyed by the Euclidean H\"ormander calculus (in fact, in a later work with Wirth \cite{RTW}, they showed that 
the calculus in \cite{ruzh+turunen_bk2010} leads to the usual H\"ormander operator classes on $\bR^{n}$ extended to compact connected manifolds).
Their approach is valid for any compact Lie group
whereas 
the calculus of \cite{bahouri+FK+gallagher_bk2012} 
is very specific to the Heisenberg group.
The crucial and new ingredient  in the definition of symbol classes in \cite{ruzh+turunen_bk2010} was defining \emph{difference operators} 
in order to replace the Euclidean derivatives in the Fourier variables.
These difference operators allow expressing the
pseudo-differential behaviour directly on the group.

In our present investigations, we build upon this notion
to study operators in the nilpotent setting.
However it is not possible to extend readily  the results of the compact case developed in
 \cite{ruzh+turunen_bk2010}
  to the nilpotent context.
 Some technical difficulties appear  because, 
 for example, the dual of $G$ is no longer discrete
 and
 the unitary irreducible representations  are almost all infinite dimensional.
More problematically there is no Laplace-Beltrami operator
and
one expects to replace it 
by a sub-Laplacian on stratified Lie groups
or, more generally, by a positive Rockland operator $\cR$ on graded Lie groups;
such operators are not central.
Hence new technical ideas are needed to develop
a pseudo-differential calculi on graded Lie groups
using the natural quantization 
together with the notion of difference operators from \cite{ruzh+turunen_bk2010}.

The results that we have obtained in our investigation of this question so far 
were presented in the talk given by the first author at the conference \emph{Fourier analysis and pseudo-differential operators}, Aalto University, 25-30 June, 2012.
They are the following
(here $1\geq \rho\geq \delta \geq 0$):

\begin{description}

\item[(R1)] 
The symbol classes form an algebra of operators $\cup_{m\in \bR} S^m_{\rho,\delta}$  stable by taking the adjoint. 

\item[(R2)]
Let $\rho\not=0$.
The operators obtained by quantization 
from $\cup_{m\in \bR} S^m_{\rho,\delta}$ 
form an algebra of operators  $\cup_{m\in \bR} \Psi^m_{\rho,\delta}$  
stable by taking the adjoint.

\item[(R3)]
The set of operators $\cup_{m\in \bR} \Psi^m_{\rho,\delta}$
contains the left-invariant calculus.

\item[(R4)] 
 The kernels are of Calderon-Zygmund type on homogeneous Lie groups;
in particular our operators of order 0 are more singular than their Euclidean counterparts.

\item[(R5)]
If $\rho\in [0,1)$,
then the operators in $\Psi^0_{\rho,\rho}$ are continuous on $L^2(G)$.

\item[(R6)]
$(\id+\cR)^{\frac m \nu}\in \Psi^m_{1,0}$, where $\cR$ is a positive Rockland operator
of degree $\nu$, see Section \ref{SEC:RO}.

\item[(R7)]
Positive operators of the calculus satisfy sharp G{\aa}rding inequalities.

\end{description}

As a consequence from Results (R2), (R5) and (R6), if $\rho\not=0$, 
any pseudo-differential operator is continuous on the Sobolev spaces 
with the loss of derivatives being controlled by the order.
All those properties justify, from our viewpoint, 
the choice of vocabulary of pseudo-differential calculi. 

Since the conference in Aalto University of June 2012,
these results, together with their complete proof 
 and other progress made by the authors on the subject have started to be collected in the 
monograph \cite{FR}, see also \cite{RF-cras1}.

In this paper, due to the lack of space,
we will state and prove the following parts of Results (1-7). 
Result (R1) is proved in Subsection  \ref{subsec_1prop_symbols}.
The proof of Result (R2) is given in  Subsection \ref{subsec_composition}
 but relies on Result (R4) which is stated 
in  Subsection \ref{subsec_1prop_kernels} 
and  only partially proved. 
In Subsection \ref{subsec_1ex},
Result (R3) is stated and proved 
while  (R6) is stated in greater generality but not proved in this paper.
The precise statements and proofs of Results (R5) and (R7) 
can be found in \cite{FR}.

\medskip

The paper is organised as follows.
In Section \ref{sec_preliminaries}, 
we explain the precise setting of our investigation for the group,
the Sobolev spaces involved here and the group Fourier transform.
In Section \ref{sec_quantization_S},
we precise our definition of quantization and symbol classes.
In Section \ref{sec_prop_symbol_kernel_op}, 
we give the properties of the symbols 
and of the corresponding kernels and operators stated above.

\medskip

{\it Convention:} All along the paper, 
$C$ denotes a constant which may vary from line to line.
We denote by $\lceil r\rceil$ the smallest integer which is strictly greater than the real $r$.

\section{Preliminaries}
\label{sec_preliminaries}

In this section, 
 we set some notation and recall some known properties
 regarding the groups under investigation,  the Taylor expansion in this context 
 and representation theory.

\subsection{The group $G$}

Here we recall briefly the definition of graded nilpotent Lie groups
and their natural homogeneous structure.
A complete description of the notions of graded and homogeneous nilpotent Lie group may be found in \cite[Chapter 1]{folland+stein_bk82}.

We will be concerned with graded Lie groups $G$
which means that $G$ is a connected and simply connected 
Lie group 
whose Lie Lie algebra $\mathfrak g$ 
admits an $\bN$-gradation
$\mathfrak g= \oplus_{\ell=1}^\infty \mathfrak g_{\ell}$
where the $\mathfrak g_{\ell}$, $\ell=1,2,\ldots$ 
are vector subspaces of $\mathfrak g$,
almost all equal to $\{0\}$
and satisfying 
$[\mathfrak g_{\ell},\mathfrak g_{\ell'}]\subset\mathfrak g_{\ell+\ell'}$
for any $\ell,\ell'\in \bN$.
This implies that the group $G$ is nilpotent.
Examples of such groups are the Heisenberg group and more generally any stratified groups (which by definition correspond to the case $\fg_1$ generating the full Lie algebra $\fg$).

Let $\{X_1,\ldots X_{n_1}\}$ be a basis of $\mathfrak g_1$ (this basis is possibly reduced to $\{0\}$), let
$\{X_{n_1+1},\ldots,  X_{n_1+n_2}\}$ a basis of $\mathfrak g_2$
and so on, so that we obtain a basis 
$X_1,\ldots, X_n$ of $\mathfrak g$ adapted to the gradation.
Via the exponential mapping $\exp_G : \mathfrak g \to G$, we   identify 
the points $(x_{1},\ldots,x_n)\in \bR^n$ 
 with the points  $x=\exp_G(x_{1}X_1+\cdots+x_n X_n)$ in $G$.
Consequently we allow ourselves to denote by $C(G)$, $\cD(G)$ and $\cS(G)$ etc...
the spaces of continuous functions, of smooth and compactly supported functions or 
of Schwartz functions on $G$ identified with $\bR^n$.
This basis also leads to a corresponding Lebesgue measure on $\mathfrak g$ and the Haar measure $dx$ on the group $G$.

The coordinate function $x=(x_1,\ldots,x_n)\in G\mapsto x_j \in \bR$
is denoted by $x_j$.
More generally we define for every multi-index $\alpha\in \bN_0^n$,
$x^\alpha:=x_1^{\alpha_1} x_2 ^{\alpha_2}\ldots x_{n}^{\alpha_n}$, 
as a function on $G$.
Similarly we set
$X^{\alpha}=X_1^{\alpha_1}X_2^{\alpha_2}\cdots
X_{n}^{\alpha_n}$ in the universal enveloping Lie algebra $\fU(\fg)$ of $\mathfrak g$.

For any $r>0$, 
we define the  linear mapping $D_r:\mathfrak g\to \mathfrak g$ by
$D_r X=r^\ell X$ for every $X\in \mathfrak g_\ell$, $\ell\in \bN$.
Then  the Lie algebra $\mathfrak g$ is endowed 
with the family of dilations  $\{D_r, r>0\}$
and becomes a homogeneous Lie algebra in the sense of 
\cite{folland+stein_bk82}.
 The weights of the dilations are the integers $\upsilon_1,\ldots, \upsilon_n$ given by $D_r X_j =r^{\upsilon_j} X_j$, $j=1,\ldots, n$.
 The associated group dilations are defined by
$$
r\cdot x
:=(r^{\upsilon_1} x_{1},r^{\upsilon_2}x_{2},\ldots,r^{\upsilon_n}x_{n}),
\quad x=(x_{1},\ldots,x_n)\in G, \ r>0.
$$
In a canonical way  this leads to the notions of homogeneity for functions and operators.
For instance
the degree of homogeneity of $x^\alpha$ and $X^\alpha$,
viewed respectively as a function and a differential operator on $G$, is 
$[\alpha]=\sum_j \upsilon_j\alpha_{j}$.
Indeed, let us recall 
that a vector of $\mathfrak g$ defines a left-invariant vector field on $G$ 
and more generally 
that the universal enveloping Lie algebra of $\mathfrak g$ 
is isomorphic with the left-invariant differential operators; 
we keep the same notation for the vectors and the corresponding operators. 

Recall that a \emph{homogeneous norm} on $G$ is a continuous function $|\cdot| : G\rightarrow [0,+\infty)$ homogeneous of degree 1
on $G$ which vanishes only at 0.
Any homogeneous norm satisfies a triangular inequality up to a constant.
Any two homogeneous norms are equivalent.
 For example
 \begin{equation}
 \label{eq_norm_||nuo}
|x|_{\nu_o}:=\left(\sum_{j=1}^n x_j^{2\frac{\nu_o}{\upsilon_j}}\right)^{\frac1{2\nu_o}}
\end{equation}
with $\nu_o$ a common multiple to the weights $\upsilon_1,\ldots,\upsilon_n$.
 
 Various aspects of analysis on $G$ can be developed in a comparable way with the Euclidean setting 
\cite{coifman+weiss-LNM71}, sometimes replacing the topological dimension 
$$
n:=\dim G =\sum_{\ell=1}^\infty\dim \fg_\ell
 ,
$$
of the group $G$ by its homogeneous dimension
$$
Q:=\sum_{\ell=1}^\infty \ell \dim \fg_\ell
=  \upsilon_1 +\upsilon_2 +\ldots +\upsilon_n 
 .
$$

\subsubsection{Taylor expansions on $G$} 
\label{subsubsec_Taylor}
In the setting of graded Lie groups one can obtain the left or right mean value theorem and left or right Taylor expansions adapted to the homogeneous structure \cite[Theorem 1.42]{folland+stein_bk82}. Let us give the statement for left invariance. We will need the following definition: 
the (left) Taylor polynomial of homogeneous degree $M$ of a function $f\in C^{M+1}(G)$ at a point $x\in G$ is by definition the polynomial $P_{x,M}^{(f)}$ satisfying
$$
X^\alpha P_{x,M}^{(f)} (0)=
\left\{\begin{array}{ll}
X^\alpha f(x)
& \mbox{whenever}\quad \alpha\in \bN^n \ \mbox{with} \ [\alpha]\leq M ,
\\
0 & \mbox{if} \ [\alpha]>M
.
\end{array}\right.
$$
We also define the remainder to be
$$
R_{x,M}^{(f)}(z):=f(xz) -P_{x,M}^{(f)}(z).
$$

\begin{proposition}[Mean value and Taylor expansion \cite{folland+stein_bk82}]
\label{prop_Taylor}
Let us fix a homogeneous norm $|\cdot |$ on $G$.
\begin{enumerate}
\item (Mean value property) There exist positive group constants $C_0$ and $b$ such that for any function $f\in C^1(G)$, we have
$$
|f(xy) -f(x)| \leq C_0 \sum_{j=1}^n |y|^j \sup_{|z|\leq b |y|} |X_j f(xz)|.
$$
\item (Taylor expansion) For each $M\in \bN_0$ there exist positive group constants $C_M$ such that for any function $f\in C^{M+1}(G)$, we have
$$
\forall y\in G\qquad 
|R_{x,M}^{(f)}(y)| \leq C_M 
\sum_{\substack{[\alpha] > M\\
|\alpha| \leq \lceil M\rceil}}
|y|^{[\alpha]} \sup_{|z|\leq b^{M+1} |y|} |X^\alpha f(xz)|
,
$$
 where $\lceil M\rfloor :=\max\{ |\beta|: \beta\in \bN_0^n\ 
 \mbox{with}\ [\beta]\leq M\}$.
 \end{enumerate}
\end{proposition}

The control can be improved in the stratified case
(again see \cite{folland+stein_bk82})
 but we  present here the more general case of the graded groups.

\begin{remark}\label{rem_vector_valued_taylor}
Proposition \ref{prop_Taylor}
extends easily to functions which are vector valued in a Banach space, replacing the modulus by operator norms.
\end{remark}

The Taylor polynomials can be described in the following way.
Let $(q_\alpha)_{\alpha\in \bN^n}$ be the basis of  polynomials 
obtained from  by the duality 
$\langle X,p\rangle:=Xp(0)$ where $X\in \fU(\fg)$ and $p$ is a polynomial.
This means that the $q_\alpha$'s are the polynomials satisfying 
\begin{equation}
\label{def_qalpha}
\forall \alpha,\beta\in \bN^n_0\qquad
X^{\beta} q_{\alpha}(0) 
=
\left\{\begin{array}{ll}
0&\mbox{if}\ \alpha \not = \beta
\\
1&\mbox{if}\ \alpha = \beta
\end{array}
\right. .
\end{equation}
We can then write the Taylor polynomial as
$$
P_{x,M}^{(f)} =\sum_{[\alpha]\leq M} X^\alpha f(x) q_\alpha 
.
$$

We will need the following properties of the polynomials $(q_\alpha)$
defined via \eqref{def_qalpha}:
\begin{lemma}
\label{lem_property_qalpha}
\begin{itemize}
\item 
Each polynomial  $q_\alpha$ is homogeneous  of degree
$[\alpha]$.
Moreover, $(q_\alpha)_{[\alpha]=d}$ is a basis of the space of homogeneous polynomials of degree $d$.
\item  
For any $\alpha_1,\alpha_2\in \bN_0^n$,
the polynomial
$q_{\alpha_1}q_{\alpha_2}$ can be written as a linear combination  
of $q_{\alpha}$ with $[\alpha]=[\alpha_1]+[\alpha_2]$.
\item For any $\alpha\in \bN_0^n$ and $x,y\in G$,
$$
q_\alpha(xy) = \sum_{[\alpha_1] + [\alpha_2] =[\alpha] } c_{\alpha_1,\alpha_2} q_{\alpha_1}(x) q_{\alpha_2}(y) ,
$$
where the coefficients $c_{\alpha_1,\alpha_2}$ are real and, moreover,
$$
c_{\alpha_1,0}=\left\{\begin{array}{ll}
1&\mbox{if}\ \alpha_1 =\alpha\\
0&\mbox{otherwise}
\end{array}\right.
\quad\mbox{and}\quad
c_{0,\alpha_2}=\left\{\begin{array}{ll}
1&\mbox{if}\ \alpha_2 =\alpha\\
0&\mbox{otherwise}
\end{array}\right.
.
$$
\end{itemize}
\end{lemma}
\begin{proof}
Clearly $(q_\alpha)_{[\alpha]=d}$ is the dual basis of  $(X^\beta)_{[\beta]=d}$.
 The first point follows.
 The second point is a direct consequence of the first together with the homogeneity of $q_{\alpha_1}q_{\alpha_2}$.
 The third point is a consequence of the homogeneity in $x$ and in $y$ and of 
the Baker-Campbell-Hausdorff formula. 
\end{proof}

\subsection{The unitary dual and the group Fourier transform}
\label{SEC:udft}

We denote by $\Gh$ the unitary dual of the group $G$, that is, 
the set of (strongly continuous) unitary irreducible representations modulo unitary equivalence. 
We will often identify a  unitary irreducible representation $\pi$ of $G$ and its equivalence class;
we denote the representation Hilbert space by $\cH_\pi$
and the subspace of smooth vectors by $\cH_\pi^\infty$.

The group Fourier transform of a function $f\in L^1(G)$ 
at $\pi\in \Gh$ is the bounded operator $\wh f(\pi)$
(sometimes this will be also denoted by $\pi(f)$ for longer expressions) on $\cH_\pi$ given by
$$
(\wh f(\pi) v_1, v_2)_{\cH_\pi} := \int_G f(x) (\pi(x)^* v_1, v_2)_{\cH_\pi} dx
,\quad v_1,v_2\in \cH_\pi .
$$
One can readily see the equality $\wh{f_1*f_2}(\pi)=\wh f_2(\pi) \wh f_1(\pi)$.

The group Fourier transform of a vector $X\in \fg$ 
at $\pi\in \Gh$ is the operator $\pi(X)$ on $\cH_\pi^\infty$ given by
$$
(\pi(X) v_1,v_2)_{\cH_{\pi}} := \partial_{s=0} (\pi(e^{sX} )v_1,v_2)_{\cH_\pi}
,\quad v_1,v_2\in \cH_\pi^\infty .
$$
Setting $\pi(X^\alpha)=\pi(X)^\alpha$,
this yields the definition of the group Fourier transform of any element of $\fU(\fg)$. 
With this notation we have for any $\alpha\in \bN_0^n$,
$$
\wh{X^\alpha f}(\pi)=\pi(X)^\alpha \wh f(\pi)=\pi(X^\alpha)\wh f(\pi)
\ \mbox{and}\ 
\wh{\tilde X^\alpha f}(\pi)=\wh f(\pi)\pi(X)^\alpha =\wh f(\pi)\pi(X^\alpha )
\, ,
$$
with the convention for $\alpha=0$ that $\pi(X^0)=\pi(I)=I =\pi(X)^0$.

\medskip

The  properties above help in the systematic computations of certain expressions; for example 
we see $\pi( \{Xf_1\} *f_2)=\pi(f_2)\pi(X)\pi(f_1)=\pi(\tilde X f_2)\pi(f_1)$
and this is coherent with the direct and more tedious computation $\{Xf_1\} *f_2=f_1 * \{\tilde X f_2\}$.

\subsection{A positive Rockland operator $\cR$}
\label{SEC:RO}

We choose $\cR$  a positive (left) Rockland operator of homogeneous degree $\nu$.
Let us recall that being a Rockland operator means that  
$\cR$ is a differential operator on $G$ 
which is left-invariant and homogeneous of degree $\nu$
and such that for every non-trivial irreducible representation $\pi$ of $G$, 
the operator
$\pi(\cR)$ is injective on smooth vectors
(see Section
\ref{SEC:udft} for the definition of $\pi(\cR)$);
being positive means
$$
\forall f\in \cS(G) \qquad
(\cR f,f)_{L^2(G)} \geq 0 
 .
$$
Here as usual
$$
(f_1,f_2)_{L^2(G)} =\int_G f_1(x) \overline{f_2(x)} dx .
$$

In the stratified case, we choose $\cR= - \cL$ 
where $\cL$ is the sub-Laplacian 
$\sum_{i=1}^{n_1} X_i^2$ (and so $\nu=2$).
In the graded case, 
if $\nu_o$ is any common multiple of the weights 
$\upsilon_1,\ldots, \upsilon_n$,
then 
\begin{equation}
\label{ex_RO}
\sum_{j=1}^n
(-1)^{\frac{\nu_o}{\upsilon_j}}
 c_j X_j^{2\frac{\nu_o}{\upsilon_j}}
\quad\mbox{and}\quad
\sum_{j=1}^n
c_j X_j^{4\frac{\nu_o}{\upsilon_j}}
\quad\left(\mbox{with}\ c_j>0\right),
\end{equation}
 are positive Rockland operators of degree $2\nu_o$ and $4\nu_o$ respectively.

By the celebrated result of Helffer and Nourrigat \cite{helffer+nourrigat-79}, 
any  Rockland operator  is hypoelliptic and satisfies subelliptic estimates. 
Furthermore, by \cite[ch. 4.B]{folland+stein_bk82},
any positive Rockland operator $\cR$,
as a differential operator defined on $\cD(G)$, 
admits   an essentially self-adjoint extension on $L^2(G)$
for which we keep the same notation $\cR$.
Let $E$ denote its spectral measure.
For any measurable function $\phi$ on  $[0,\infty)$, 
we define the operator
$$
\phi(\cR) := \int_0^\infty \phi(\lambda) dE_\lambda
,
$$
which is invariant under left-translation. 
If it maps continuously $\cS(G)\rightarrow \cS'(G)$ (for example if $\phi$ is bounded),
by the Schwartz kernel theorem, it is a convolution operator 
with kernel $\phi(\cR)\delta_o \in \cS'(G)$, that is,
$$
\phi(\cR) f = f \ * \ \phi(\cR)\delta_o ,\quad f\in \cS(G)
.
$$
Recall that the group convolution  is defined via
$$
f_1*f_2(g) =\int_G f_1(g') f_2({g'}^{-1} g)  dg'
 , \quad f_1,f_2\in \cS(G)
.
$$

The above hypotheses ensure the following Marcinkiewicz-type properties
proved  by A. Hulanicki \cite{hula-1984}.
 \begin{proposition}[Hulanicki]
 \label{prop_hula}
For any $\alpha,\beta \in \bN_0^n$,
there exists $k=k_{\alpha,\beta}\in \bN_0$ and $C=C_{\alpha,\beta}>0$
such that for any $\phi\in C^\infty([0,\infty))$ we have
$$
\|x^\alpha X^\beta \phi(\cR)\delta_o\|_{L^1(G)} \leq C
\sup_{\lambda\geq 0,\, k_1\leq k}
 (1+\lambda)^k \partial_\lambda^{k_1} \left|\phi(\lambda)\right| 
.
$$
By this we mean that if the supremum in the right hand side is finite, then 
the distribution $x^\alpha X^\beta \phi(\cR)\delta_o$ coincides with an integrable function and
the  inequality holds.
\end{proposition}

\begin{remark}\label{rem_cq_hula}
Consequently if $\phi\in \cS(\bR)$ is Schwartz, then the kernel $\phi(\cR)\delta_o$ is also Schwartz on $G$, i.e. $\phi(\cR)\delta_o\in \cS(G)$. 
\end{remark}

The proof of Proposition \ref{prop_hula}
relies on using the   $\cR$-\emph{heat kernel} $h_t$,
defined as the kernel of 
$\exp (-t\cR)$ for each $t>0$.
In \cite{folland+stein_bk82}, it is proved that 
the function $h=h_1$ is Schwartz and that
$$
h_t(x) =t^{-\frac Q \nu} h(t^{-\frac 1\nu} x)
.
$$

As in the Euclidean or stratified cases (see \cite{folland-1975}),
we can define \emph{Bessel potentials} associated with a positive Rockland operator $\cR$ of degree $\nu$
via the integral 
$$
\cB_a (x)= \frac 1{\Gamma(\frac a \nu)}
\int_0^\infty t^{\frac a \nu  -1 } e^{-t} h_t(x) dt
.
$$
Indeed for $a\in \bC$ with $\Rep a>0$,
this integral converges absolutely for $x\not=0$ and defines the Bessel potential
$\cB_a\in C^\infty(G\backslash \{0\})$ which satisfies:
$$
\|\cB_a\|_{L^1(G)} \leq \frac{\Gamma (\Rep \frac a \nu)}{|\Gamma (\frac a \nu)|} \|h\|_{L^1}<\infty
 , \quad a\in \bC, \ \Rep a>0
.
$$
Using the properties of semigroup of $e^{-t\cR}$,
one obtains that 
$\|\cB_a\|_{L^2(G)}$ is square integrable if $\Rep a >Q/2$.
The Bessel potential is the convolution kernel 
of the $L^2(G)$-bounded left-invariant operator $(\id+\cR)^{-a/\nu}$
and 
of the $L^2(G)$-bounded right-invariant operator $(\id+\tilde\cR)^{-a/\nu}$, so that we have
$$
(\id+\cR)^{-a/\nu} f= f* \cB_a ,\quad
(\id+\tilde\cR)^{-a/\nu} f= \cB_a*f
,\quad f\in L^2(G) .
$$

\subsection{Sobolev spaces}
\label{subsec_sobolev_spaces}

For $a\geq 0$ and $\cR$ a positive Rockland operator,
we define  the $\cR$-Sobolev spaces as the domain of $(\id+\cR)^{\frac a\nu}$ , that is,
$$
L^2_a(G) = \{ f\in L^2(G) , \ (\id+\cR)^{\frac a\nu} f\in L^2(G) \} 
.
$$
For $a<0$, $L^2_a(G)$ is the completion of $L^2(G)$ for the norm 
$f\mapsto \|(\id+\cR)^{\frac a\nu} f\|_{L^2(G)}$.
It is easy to see that for any $a\in \bR$, the Sobolev space $L^2_a(G)$ is a Hilbert space for the norm
$$
\|f\|_{L^2_a(G)}:= \|(\id+\cR)^{\frac a\nu} f\|_{L^2(G)}
.
$$
Adapting the stratified case \cite{folland-1975} 
(see \cite{FR}), 
one obtains:
\begin{proposition}[Sobolev spaces]
\label{prop_sobolev_spaces}
Let $\cR$ be a  positive Rockland operator of homogeneous degree $\nu_\cR$.

\begin{enumerate}
\item 
If $a\leq b$, then
$\cS(G)\subset L^2_b (G)\subset L^2_a(G) \subset \cS'(G)$
and an equivalent norm for $L^2_b(G)$ is
$f\mapsto \|f\|_{L^2_a(G)} + \|\cR ^{\frac {b-a} \nu} f\|_{L^2_a(G)}$.
\item 
If $a\in \nu_\cR\bN_0$, then an equivalent norm is given by
$f\mapsto\sum_{[\alpha]\leq a} \|X^\alpha f\|_{L^2(G)}$.
\item 
The dual space of $L^2_a(G)$ is isomorphic to $L^2_{-a}(G)$ via
the bilinear form $(f_1,f_2) \mapsto \int_G f_1 f_2 dg$. 
\item 
We have the usual property of interpolation for Sobolev spaces:
let $T$ be a linear mapping from $L^2_{a_0}(G) + L^2_{a_1}(G)$
to locally integrable functions on $G$;
we assume that $T$ maps $L^2_{a_0}(G)$ and $L^2_{a_1}(G)$
boundedly into $L^2_{b_0}(G)$ and $L^2_{b_1}(G)$, respectively.
Then $T$ extends uniquely to a bounded mapping from 
$L^2_{a_t}(G)$ to $L^2_{b_t}(G)$
with $(a_t,b_t) = t(a_0,b_0)+(1-t) (a_1,b_1)$.
\end{enumerate}
\end{proposition}

Consequently, the Sobolev spaces do not depend on the choice of operators $\cR$ as in the statement above. Such operators always exist (see \eqref{ex_RO}) and we fix one of them until the end of the paper.

From the interpolation property of Sobolev spaces
(cf Proposition \ref{prop_sobolev_spaces}), we have:
\begin{lemma}
\label{lem_cq_interpolation}
Let $\kappa \in \cS'(G)$ and $a\in \bR$.
Let $\{\gamma_n,\, n\in \bZ\}$
be a sequence of real numbers 
which tends to $\pm \infty$ as $n\to\pm \infty$.
Assume that for any $n\in \bZ$,
the operator $T_\kappa$  extends continuously 
to a bounded operator $L^2_{\gamma_n} (G)\rightarrow L^2_{a+\gamma_n}(G)$. 
Then the operator $T_\kappa$  extends continuously 
to a bounded operator $L^2_\gamma (G)\rightarrow L^2_{a+\gamma}(G)$ 
 for any $\gamma\in \bR$.
\end{lemma}

As in the Euclidean and stratified cases \cite{folland-1975}, we can prove the following Sobolev inequalities:

\begin{lemma}[Sobolev inequality]
\label{lem_sob_ineq}
If $a>Q/2$ then any function $f\in L^2_a(G)$ admits a continuous bounded representative which 
satisfies
$$
\|f\|_{L^\infty(G)} \leq C_a \|f\|_{L^2_a(G)}
,
$$
with $C_a=\|\cB_a\|_{L^2(G)}$ independent of $f$.
\end{lemma}
\begin{proof}[Sketch of the proof]
It suffices to write
$$
f=(\id+\cR)^{-\frac a\nu}  (\id+\cR)^{\frac a\nu}  f = \{(\id+\cR)^{\frac a\nu}  f\} * \cB_a
.
$$
\end{proof}

\subsection{The Plancherel Theorem and the von Neumann algebras 
$\sL_L(L^2(G))$, $\cK(G)$ and $L^\infty(\widehat G)$}

About representation theory and the Plancherel theorem,
we refer the reader to Dixmier's standard textbook \cite{dixmier_bk1969}, especially \S 18.8.

Recall that a bounded operator $A$ on a Hilbert space $\cH$ is in the Hilbert-Schmidt class whenever $\| A\|_{HS} =\sqrt{\tr \left(A^* A\right)}$ is finite.
If $f\in L^2(G) \cap L^1(G)$ then $\wh{f}(\pi)$ is a Hilbert-Schmidt operator, 
and the \emph{Plancherel formula} holds,
$$
\int_G |f(g)|^2 dg = \int_{\Gh} \| \wh f(\pi)\|_{HS}^2 d\mu(\pi)
,
$$
where $\mu$ is the Plancherel measure on $\Gh$.
The group Fourier transform extends unitarily to $L^2(G)$
and  
a square integrable function $f\in L^2(G)$ gives rise to a $\mu$-square-integrable field of Hilbert-Schmidt operators $\{\wh{f}(\pi)\}$. 
Conversely, a $\mu$-square-integrable field of Hilbert-Schmidt operators 
$\{\sigma_\pi\}$ defines a square integrable function $f$ with
$$
(f,f_1)_{L^2(G)}=\int_{\Gh} \tr \left( \sigma_\pi \ \wh{f_1}(\pi)^*\right) d\mu(\pi),\quad f_1\in L^2(G)
.
$$

Let $\sL(L^2(G))$ denote the set of 
bounded linear operators $L^2(G)\rightarrow L^2(G)$,
and let $\sL_L(L^2(G))$
be the subset formed by the operators in $\sL(L^2(G))$
which commute with the left regular representation $L(g): f\in L^2(G) \mapsto f(g^{-1} \cdot)$, $g\in G$.
Endowed with the operator norm and composition of operators, 
 $\sL_L(L^2(G))$
is a von Neumann algebra.

If $T \in \sL_L(L^2(G))$, 
then there exists a $\mu$-measurable field 
of uniformly bounded operators $\{\sigma^{(T)}_\pi\}$ 
such that for any $f\in L^2(G)$
the Hilbert-Schmidt operators $\wh{Tf}(\pi)$ and $\sigma^{(T)}_\pi \wh f(\pi)$
are equal $\mu$-almost everywhere; 
the field $\{\sigma^{(T)}_\pi\}$ is unique up to a $\mu$-negligible set.

Let $L^\infty(\widehat G)$ denote the space of $\mu$-measurable fields of uniformly bounded operators on $\widehat G$, modulo equivalence with respect to the Plancherel formula $\mu$. As is usual, we will identify such fields with their classes in $L^\infty(\widehat G)$.
We have that if $T\in \sL_L(L^2(G))$, 
then there exists a unique $\sigma\in L^\infty(\widehat G)$ as above.
Note that 
by the Schwartz kernel theorem, 
 the operator $T$ is of convolution type with kernel $\kappa\in \cD'(G)$,
$$
Tf = f* \kappa ,\quad f\in \cD(G)
.
$$
Conversely given a  field $\{\sigma_\pi\}\in L^\infty(\widehat G)$ 
there exists a unique bounded linear operator $T\in \sL_L(L^2(G))$
 satisfying $\wh{Tf}(\pi)=\sigma_\pi \wh f(\pi)$ $\mu$-almost everywhere
 for any $f\in L^2(G)$.

If $\kappa\in \cD'(G)$ is such that the corresponding convolution operator 
$f\in \cD(G)\mapsto f*\kappa $ extends to a bounded operator
$T \in \sL(L^2(G))$
then 
$T \in \sL_L(L^2(G))$
and
we abuse the notation by setting 
$\sigma^{(T)}_\pi:=\pi(\kappa)\equiv\wh\kappa(\pi)$.
We denote by $\cK(G)$ the set of such distributions $\kappa$.
It is a von  Neumann algebra isomorphic to $\sL_L(L^2(G))$
when equipped with the $*$-product $\kappa\mapsto \kappa^*$ where $\kappa^*(x)=\bar \kappa(x^{-1})$,
and  the operator norm 
$$
\|\kappa\|_* := 
\| f\mapsto f*\kappa\|_{\sL(L^2(G))}.
$$

Note that when we equip $L^\infty(\widehat G)$
with the operation 
$\sigma\mapsto \sigma^*$
and the norm
$$
\|\sigma\|_* = \sup_{\pi\in \Gh} \|\wh\sigma_\pi\|_{op},
$$
where  $\|\cdot\|_{op}$ denotes the operator norm
and the supremum is in fact the essential supremum with respect to the Plancherel measure $\mu$,
$L^\infty(\widehat G)$ becomes a  von  Neumann algebra
isomorphic with $\sL_L(L^2(G))$ and $\cK(G)$.
More precisely the group Fourier transform defined on $\cK(G)$
gives the isomorphism between $\cK(G)$ and $L^\infty(\widehat G)$.

\medskip

Throughout this paper, if $\kappa\in \cD'(G)$,
 $T_\kappa$ denotes the convolution operator
 $$
 T_\kappa: \cD(G)\ni f\mapsto f*\kappa
 ,
 $$
 and we keep the same notation for any of its continuous extensions 
 $L^2_b(G)\rightarrow L^2_a(G)$ when they exist.
With norms possibly infinite,
$\|\kappa\|_*$ is equal to the operator norm 
of $T_\kappa:L^2(G)\to L^2(G)$
by the Plancherel theorem,
 and is less than $\|\kappa\|_{L^1(G)}$.

For any $a,b\in \bR$, it is easy to see  
that $T_\kappa$ admits a continuous extension
 $L^2_b(G)\rightarrow L^2_a(G)$
if and only if 
$(I+\tilde \cR)^{-\frac b \nu}(\id+\cR)^{\frac a \nu} \kappa \in \cK(G)$,
with equality between the $L^2_b(G)\rightarrow L^2_a(G)$-operator norm and the $\cK(G)$-norm.
 In this case we may abuse the notation and write
$$
\pi(\id+\cR)^{\frac a\nu} \pi(\kappa) \pi(\id+\cR)^{-\frac b\nu} 
\quad\mbox{instead of}\quad
\pi\left((\id+\cR)^{\frac a\nu} (I+\tilde \cR)^{-\frac b\nu} \kappa \right)
.
$$ 

\section{Quantization and symbols classes}
\label{sec_quantization_S}

As recalled in Introduction, 
there exists a natural quantization  
which is valid on any Lie group of type~I.
We will present it in this section
after defining symbols 
for which this quantization makes sense 
and produces operators $\cD(G)\to\cD'(G)$
with $G$ graded Lie groups.
Moreover, the resulting operators admit 
integral representations with right convolution kernels
and these kernels play a major role in every subsequent proof.
We will also define symbol classes and give some examples of symbols.

\subsection{The symbols and their kernels}
\label{subsec_presymbol}

A \emph{symbol} is a family of operators 
$\sigma=\{\sigma(x,\pi), \ x\in G, \ \pi\in \Gh\}$
 satisfying:
\begin{enumerate}
\item for each $x\in G$, 
$\{\sigma(x,\pi), \ \pi\in \Gh\}$
is a $\mu$-measurable field of operators $\cH_\pi^\infty\rightarrow \cH_\pi$,
\item 
there exist $\gamma_1,\gamma_2\in \bR$ such that 
for any $x\in G$,
\begin{equation}
\label{field_sigma_gamma12}
\{\pi(\id+\cR)^{\gamma_1} \sigma(x,\pi) \pi(\id+\cR)^{\gamma_2}, \ \pi \in \Gh\} \in L^\infty(\widehat G)
,
\end{equation}
\item  
for any $\pi\in \Gh$ and any $u,v\in \cH_{\pi}$,
the scalar function $x\mapsto (\sigma(x,\pi) u,v)_{\cH_\pi}$
is smooth over $G$.
\end{enumerate}

\medskip

Consequently at each $x\in G$ and $\pi\in \Gh$, 
the operator $\sigma(x,\pi)$ is densely defined on $\cH_\pi$;
it is also the case for $X^\beta_x\sigma(x,\pi)$
for any $\beta\in \bN_0^n$.

The second condition implies that for each $x\in G$, the $\mu$-measurable field
\eqref{field_sigma_gamma12} correspond to a distribution $\kappa_{x,\gamma_1,\gamma_2}\in \cK(G)$ which depends smoothly on $x$;
hence  $\sigma$ corresponds to a distribution
$$
\kappa(x, \cdot)=\kappa_x:=(\id+\cR)^{-\gamma_1} (I+\tilde \cR)^{-\gamma_2} \kappa_{x,\gamma_1,\gamma_2}
,
$$
which we call its \emph{kernel}.
By injectivity of $\pi$ on $\cK(G)$, 
$\pi(X_x^\beta\kappa_x)=X_x^\beta \sigma(x,\pi)$. 
 
Examples of symbols are the symbols within the  classes $S^m_{\rho,\delta}$ defined later on.
More concrete examples of symbols which do not depend on $x\in G$ 
are $\pi(X)^\alpha$, $\alpha\in \bN_0^n$ 
or the multipliers in $\pi(\cR)$, that is, 
$\phi(\pi(\cR))$ with $\phi\in L^\infty(\bR)$ (for example).
Indeed for any $\pi\in \Gh$ 
the operator $\pi(\cR)$ is essentially self-adjoint 
\cite{hula+jenkins+ludwig-1985} 
and we denote by $E_\pi$ its spectral projection, 
hence giving a meaning to $\phi(\pi(\cR))$.
The relation between the spectral projections $E$ and $E_\pi$ of $\cR$ and $\pi(\cR)$ is
$$
\pi(\phi(\cR)f)=\phi(\pi(\cR)) \pi(f)
, \quad \phi\in L^\infty(\bR) , \ f\in L^2(G) .
$$
It is known \cite{hula+jenkins+ludwig-1985}
 that the spectrum of  $\pi(\cR)$ consists of discrete eigenvalues in $(0,\infty)$. 
This may add a further justification to using the word \emph{quantization}. 
 
\subsection{The quantization mapping $\sigma\mapsto Op(\sigma)$} 
\label{subsec_quantization}

Our quantization is analogous to the usual Kohn-Nirenberg quantization in the Euclidean setting, and has already  been noticed by Taylor \cite{Tnma},
used indirectly on the Heisenberg group \cite{Tnma,bahouri+FK+gallagher_bk2012}
and explicitly on compact Lie groups \cite{ruzh+turunen_bk2010}.
It associates an operator 
$T=Op(\sigma)$ to a symbol $\sigma$ in the following way
(with the same notation as in Subsection \ref{subsec_presymbol}). 
For any $f\in \cD(G)$ and $x\in G$, 
\begin{eqnarray*}
\int_{\Gh} \tr \left| \sigma(x,\pi)\wh f(\pi)\right| d\mu(\pi)
\leq 
\sup_{\pi\in \Gh} \| \pi(\id+\cR)^{\gamma_1} \sigma(x,\pi)  \ \pi(\id+\cR)^{\gamma_2} \|_{op}\\
\int_{\Gh} \tr \left|  \pi\left((\id+\cR)^{-\gamma_1}(I+\tilde \cR)^{-\gamma_2}   f\right)\right| d\mu(\pi)
,
\end{eqnarray*}
is finite and we can set
\begin{equation}\label{EQ:quant}
Tf(x):= \int_{\Gh} \tr \left(\pi(x)\sigma(x,\pi)\whfpi\right) d\mu(\pi)
.
\end{equation}
We have obtained a continuous linear operator $T:\cD(G)\to \cD'(G)$.
By the Schwartz kernel theorem, $T=Op(\sigma)$ has an integral kernel in the distributional sense. 
However since $\sigma$ is a symbol, 
we obtain directly, still in the distributional sense, 
the following integral representation in terms of the kernel $\kappa$ defined in Subsection \ref{subsec_presymbol},
$$
Tf(x) =f*\kappa_x(x)= \int_G f(y) \kappa(x, y^{-1}x) dy
.
$$

For example, the symbol $\sigma$ given by the identity operator on each space $\cH_\pi$ is associated with the identity operator on $G$; its kernel is the Dirac measure at 0 denoted by $\delta_0$ (independent of the point $x\in G$).
More generally, for any $\alpha\in \bN_0^n$,
 the symbol $\pi(X)^\alpha$ is associated with the operator $X^\alpha$
 with kernel $(-1)^{|\alpha|}X^\alpha \delta_0$ defined in the sense of distributions via
 $$
 \int_G f(g) (-1)^{|\alpha|} X^\alpha \delta_0(g) dg = 
 \int_G X^\alpha  f(g)  \delta_0(g) dg = 
 X^\alpha f(0) .
 $$

It is easy to see that the quantization mapping $\sigma\mapsto T=Op(\sigma)$
is 1-1 and linear.

Before defining symbol classes, 
we need to define difference operators.

 \subsection{Difference operators}
 Difference operators were defined on compact Lie groups in  
 \cite{ruzh+turunen_bk2010}, as acting on Fourier coefficients.
Its adaptation to our setting leads us to define difference operators on $L^\infty(\widehat G)$ 
viewed as fields.
More precisely for any $q\in C^\infty(G)$, we set
$$
\Delta_q\wh f(\pi):=\wh{qf}(\pi) =\pi(qf) .
$$
This defines an operator $\Delta_q$ 
with domain
$\dom(\Delta_q):=\cF_G \{f\in \cK(G), \ qf\in \cK(G)\}$,
and more generally   $\pi(\id+\cR)^{-\gamma_1}\pi(I+\tilde \cR)^{-\gamma_2}\dom(\Delta_q)$ for any $\gamma_1,\gamma_2\in \bR$.
Note that in general, it is not possible to define an operator $\Delta_q$ on each $\cH_\pi$; 
this can be seen quite easily by considering
the multiplication by the central variable
on  the Heisenberg group for example.

The \emph{difference operators} are
$$
\Delta^\alpha :=\Delta_{\tilde q_{\alpha}}, \quad \alpha\in \bN_0^n,
$$
where $\tilde q_{\alpha}(x)=q_\alpha(x^{-1})$
and the $q_\alpha$'s were defined via \eqref{def_qalpha}.

Lemma \ref{lem_property_qalpha} implies that 
$\Delta^{\alpha_1}\Delta^{\alpha_2}$
is a linear combination
of $\Delta^{\alpha}$ with $[\alpha]=[\alpha_1]+[\alpha_2]$.
Furthermore, we have
 \begin{eqnarray*}
&&\tilde q_\alpha(x) \ f_2*f_1(x) 
=
\int_G q_\alpha(x^{-1}  y \ y^{-1}) \ f_2(y)\  f_1(y^{-1} x) \ dy
\\
&&\qquad=
\!\!\!\!\!\!\sum_{[\alpha_1] + [\alpha_2] =[\alpha] }\!\!\!\!\!\!
 c_{\alpha_1,\alpha_2} 
\int_G 
 q_{\alpha_2}(y^{-1})
  f_2(y)\ q_{\alpha_1}(x^{-1}  y) f_1(y^{-1} x) \ dy
\\
&&\qquad=
\!\!\!\!\!\!\sum_{[\alpha_1] + [\alpha_2] =[\alpha] }\!\!\!\!\!\!
c_{\alpha_1,\alpha_2} \
(\tilde q_{\alpha_2}f_2)*(\tilde q_{\alpha_1}f_1) ,
\end{eqnarray*}
and
we get the \emph{Leibniz formula}:
 \begin{equation}
\label{formula_leibniz}
\Delta^\alpha \left(\wh{f_1}(\pi)\wh{f_2}(\pi)\right)
= \sum_{[\alpha_1] + [\alpha_2] =[\alpha] } c_{\alpha_1,\alpha_2} 
\ \Delta^{\alpha_1}\wh{f_1}(\pi)\ \Delta^{\alpha_2} \wh{f_2}(\pi).
\end{equation}


 The idea of difference operators appear naturally 
 when considering operators on the torus $\bT^{n}$.
 In this case  one recovers forward and backward difference
 operators on the lattice $\bZ^{n}$.
 Difference operators were systematically defined and studied 
 on compact Lie groups in \cite{ruzh+turunen_bk2010}.
 On the Heisenberg group, expressions of a related nature 
 were used to describe the Schwartz space in \cite{geller} 
  and with a hypothesis of unitary invariance in  \cite{BJR}.
 
\subsection{The symbol classes $S^m_{\rho,\delta}$}

\begin{definition}
Let $m,\rho,\delta\in \bR$ with $1\geq \rho\geq \delta\geq 0$ and $\delta\not=1$.
A symbol $\sigma$ is a {\em symbol of order $m$ and of type~$(\rho,\delta)$} 
whenever,
for each $\alpha,\beta\in \bN_0^n$ and $\gamma\in \bR$, 
the field
$$
\{\pi(\id+\cR)^{\frac{\rho [\alpha]-m -\delta[\beta] +\gamma}\nu }
X_x^\beta\Delta^\alpha \sigma(x,\pi) 
\pi(\id+\cR)^{-\frac{\gamma}\nu }, \ \pi\in \Gh\}
,
$$
is in $L^\infty(\widehat G)$ uniformly in $x\in G$;
this means that we have
$$
\sup_{\pi\in \Gh, \, x\in G}
\|\pi(\id+\cR)^{\frac{\rho [\alpha]-m -\delta[\beta] +\gamma}\nu }
X_x^\beta\Delta^\alpha \sigma(x,\pi) 
\pi(\id+\cR)^{-\frac{\gamma}\nu }\|_{op} 
=C_{\alpha,\beta,\gamma} <\infty
.
$$
(The supremum over $\pi$ is in fact the essential supremum with respect to the Plancherel measure $\mu$.)

The \emph{symbol class}  $S^m_{\rho,\delta}$ is the set of symbol of order $m$ and of type~$(\rho,\delta)$.

We also define $S^{-\infty}_{\rho,\delta}=\cap_{m\in \bR}S^m_{\rho,\delta}$ the class of  smoothing symbols.
\end{definition}

Let us make some comments on this definition:
\begin{enumerate}
\item 
In the abelian case,
that is, $\bR^n$ endowed with the addition law
and $\cR=-\cL$, $\cL$ being the Laplace operator, 
$S^m_{\rho,\delta}$ boils down to the usual H\"ormander class.
In the case of compact Lie groups
with $\cR=-\cL$, $\cL$ being the Laplace-Beltrami operator, 
a similar definition leads to the one considered in \cite{ruzh+turunen_bk2010}
 since  the operator $\pi(\id+\cR)$ is scalar.
However here, in the case of non-abelian graded groups, 
the operator $\cR$ can not have a scalar 
Fourier transform.
\item
The presence of the parameter $\gamma$ is required to prove that the space of symbols  $\cup_{m\in \bR} S^m_{\rho,\delta}$ form an algebra of operators later on. 
\item\label{rem_gamma_countable} The conditions on $\alpha$ and $\beta$ are of countable nature and it is also the case for $\gamma$. 
Indeed, 
by Lemma \ref{lem_cq_interpolation}, it suffices to prove the property above for one sequence $\{\gamma_n,\, n\in \bZ\}$ 
with 
$\gamma_n\underset  {n\rightarrow \pm\infty}\longrightarrow\pm\infty$.
\item A symbol class $S^m_{\rho,\delta}$ is a vector space.
And we have the inclusions
$$
m_1\leq  m_2,\quad \delta_1\leq\delta_2,\quad
\rho_1\geq\rho_2
\quad \Longrightarrow \quad
S^{m_1}_{\rho_1,\delta_1}\subset S^{m_2}_{\rho_2,\delta_2}
.
$$
\item If $\rho\not=0$, 
we will show in Subsections \ref{subsec_1prop_kernels}
and \ref{subsec_composition} 
that we obtain an algebra of operators 
with smooth kernels $\kappa_x$ away from the origin.

\end{enumerate}

If $\sigma$ is a symbol and $a,b,c\in [0,\infty)$, we set
$$
\|\sigma(x,\pi)\|_{S^m_{\rho,\delta},a,b,c}: = 
\!\!\!\!\!\! \sup_{\substack{|\gamma|\leq c \\ [\alpha]\leq a,\, [\beta]\leq b}}\!\!\!\!\!\!
\|\pi(\id+\cR)^{\frac{\rho [\alpha]-m -\delta[\beta] +\gamma}\nu }
X_x^\beta\Delta^\alpha \sigma(x,\pi) 
\pi(\id+\cR)^{-\frac{\gamma}\nu }\|_{op} 
\, ,
$$
and 
$$
\|\sigma\|_{S^m_{\rho,\delta},a,b,c} := \sup_{x\in G, \, \pi\in \Gh} \|\sigma(x,\pi)\|_{S^m_{\rho,\delta},a,b,c}
.
$$

It is a routine exercise to check that for 
any $m\in \bR$, $\rho,\delta\geq 0$,
 the functions 
$\|\cdot\|_{S^m_{\rho,\delta},a,b,c}$, 
$a,b,c\in [0,\infty) $, are semi-norms over the vector space $S^m_{\rho,\delta}$.
Furthermore, 
with Comment \ref{rem_gamma_countable}  above, 
taking $a,b,c$ as non-negative integers, they endow  
$S^m_{\rho,\delta}$ of a structure of Fr\'echet space.
The class of smoothing symbols is then equipped with the topology of projective limit.

The pseudo-differential operators of order $m\in \bR\cup\{-\infty\}$ and type~$(\rho,\delta)$ are obtained by quantization from the symbols of the same order and type, that is,
$$
\Psi^m_{\rho,\delta}:=Op(S^m_{\rho,\delta})
,
$$
with the quantization given by \eqref{EQ:quant}.
They inherit a structure of topological vector space from the classes of symbols,
$$
\|Op(\sigma)\|_{\Psi^m_{\rho,\delta},a,b,c}:=
\|\sigma\|_{S^m_{\rho,\delta},a,b,c}
.
$$

It is not difficult to see from the computations in Subsection \ref{subsec_quantization} 
that any operator $Op(\sigma)$ is a continuous operator $\cD(G)\rightarrow \cD'(G)$;
in fact, we can show 
that $T$ is continuous $\cS(G)\rightarrow \cS(G)$ 
but the complete proof which uses Theorem \ref{thm_composition} and Proposition \ref{prop_multipliers} 
can be found in  \cite{FR}.

The type~$(1,0)$ is thought of as the basic class of symbols and the 
types~$(\rho,\delta)$ as their generalisations,
the limitation on the parameters $(\rho,\delta)$ 
coming from reasons similar to the ones in the Euclidean settings.
For type~$(1,0)$, we set $S^m:=S^m_{1,0}$, $\Psi^m:=\Psi^m_{1,0}$ 
and,
$$
\|\sigma(x,\pi)\|_{S^m_{1,0},a,b,c}=\|\sigma(x,\pi)\|_{a,b,c}
, \
\|\sigma\|_{S^m_{1,0},a,b,c}=\|\sigma\|_{a,b,c} , \ \mbox{etc}\ldots
$$

Before proving that $\cup_{m\in \bR} S^m_{\rho,\delta}$ and 
$\cup_{m\in \bR} \Psi^m_{\rho,\delta}$ are stable by composition,
let us give some examples.

\subsection{First examples}
\label{subsec_1ex}

As it should be, $\cup_{m\in \bR} \Psi^m$
contains the calculus of left invariant differential operators.
More precisely the following lemma implies that
$\sum_{[\beta]\leq m} c_\beta X^\beta \in \Psi^m$.
The coefficients $c_\alpha$ here are constant 
and it is easy to relax  this condition 
with each function $c_\alpha$ being smooth and bounded as well as all its derivatives.

\begin{lemma}
\label{lem_Xbeta_PsiDO}
For any $\beta_o\in \bN_0^n$, 
the operator $X^{\beta_o}=Op(\pi(X)^{\beta_{0}})$  is in $\Psi^{[\beta_o]}$.
\end{lemma}
\begin{proof}
For any $\alpha\in \bN_0^n$, 
in the sense of distributions,
$$
\int_G f(g)  (\tilde q_\alpha (-1)^{|\beta_o|}X^{\beta_{0}} \delta_0)(g) dg
=
\int_G X^{\beta_o} \left\{\tilde q_\alpha(g) f(g) \right\}  \delta_0(g) dg
,
$$
is always zero if $[\alpha]<[\beta_o]$ 
or $[\alpha]=[\beta_o]$ with $\alpha\not=\beta_o$.
If $[\alpha]>[\beta_o]$ or $[\alpha]=[\beta_o]$ with $\alpha=\beta_o$,
then it is equal to $X^{\beta_o-\alpha}f$ up to some constant $c_{\alpha,\beta_o}\in \bR$.
Moreover, in the latter case, we get
\begin{eqnarray*}
\| f* (\tilde q_\alpha (-1)^{|\beta_o|}X^{\beta_{0}} \delta_0)\|_{L^2_{[\alpha]-[\beta_o] +\gamma}}
&=&
|c_{\alpha,\beta_{0}}|
\|X^{\beta_{0}-\alpha} f\|_{L^2_{[\alpha]-[\beta_o] +\gamma}(G)}
\\
&\leq& C_{\alpha,\beta_o} \|f\|_{L^2_\gamma(G)} .
\end{eqnarray*}
This shows 
$\|\pi(\id+\cR)^{\frac{[\alpha]-[\beta_o] +\gamma}\nu }
\Delta^\alpha \pi(X)^{\beta_o}
\pi(\id+\cR)^{-\frac{\gamma}\nu }\|_{op}\leq C_{\alpha,\beta_o}$.
\end{proof}

An example of smoothing operator is given by convolution with a Schwartz function:
\begin{lemma}
\label{lem_ex_smoothing_op}
If $\kappa\in \cS(G)$ then $T_\kappa \in \Psi^{-\infty}$.
Furthermore, the mapping $\cS(G)\ni \kappa\mapsto T_\kappa \in \Psi^{-\infty}$ is continuous.
\end{lemma}

\begin{proof}
For any $\kappa\in \cS(G)$ and $a\geq 0$,
we have $(\id+\cR)^a\kappa \in L^1(G)$.
Indeed, it is true if $a\in \bN_0$; 
if $a\not\in \bN_0$, 
then writing
$$
(\id+\cR)^a\kappa = \left\{(\id+\cR)^{\lceil a\rceil }\kappa \right\}
* \cB_{a-\lceil a\rceil}
,
$$
we get
$$
\|(\id+\cR)^a\kappa\|_{L^1(G)} 
\leq
\| (\id+\cR)^{\lceil a\rceil }\kappa\|_{L^1(G)} 
\| \cB_{a-\lceil a\rceil}\|_{L^1(G)} 
.
$$
We have also the same property for $\tilde \cR$ by adapting the proof above.

Let $m\in \bR$.
For any $\gamma\in \bR$ and $\alpha\in \bN_0^n$
such that $\gamma$ and $[\alpha]-m+\gamma$ are of the same sign, we have
\begin{eqnarray*}
&&\sup_{\pi\in \Gh}
\|\pi(\id+\cR)^{\frac{[\alpha]-m +\gamma}\nu }
\Delta^\alpha \pi(\kappa)
\pi(\id+\cR)^{-\frac{\gamma}\nu }\|_{op}\leq 
\\&& 
\left\{\begin{array}{ll}
\|\cB_\gamma\|_{L^1(G)} 
\|(\id+\cR)^{\frac{[\alpha]-m+\gamma} \nu}\tilde q_\alpha \kappa\|_{L^1(G)} 
& \mbox{if} \ \gamma, [\alpha]-m+\gamma\geq 0 ,\\
\|(\id+\tilde \cR)^{-\frac{\gamma} \nu}\tilde q_\alpha\kappa\|_{L^1(G)} 
\|\cB_{-([\alpha] -m+\gamma)}\|_{L^1(G)} 
& \mbox{if} \ \gamma,[\alpha]-m+\gamma\leq 0 .\\
\end{array}\right.
\end{eqnarray*}
It is now clear that $T_\kappa\in \Psi^m$ and that any semi-norm 
$\|T\|_{\Psi^m,a,b,c}$ is controlled by some Schwartz semi-norm of $\kappa$.
\end{proof}

By Lemma \ref{lem_ex_smoothing_op} and Remark \ref{rem_cq_hula},
 if $\phi\in \cS(\bR)$ then $\phi(\cR)\in \Psi^{-\infty}$.
This last consequence could also be obtained via the next example.

\medskip

The $\cR$-multipliers in the following class of functions yields operators in the calculus.
We consider  the space $\cM_m$ of smooth functions $\phi$ on $[0,\infty)$ 
satisfying for every $k\in\bN_0$:
$$
\| \phi\|_{\cM_m , k} := \sup_{\lambda \geq 0,\, k_1\leq k}
\left| (1+\lambda)^{-m+k_1} \partial_\lambda^{k_1} \phi(\lambda)\right| \ 
<\infty
.
$$
An important example is $\phi(\lambda)=(1+\lambda)^m$, $m\in \bR$.

\begin{proposition}
\label{prop_multipliers}
Let $m\in \bR$ and $\phi\in \cM_{\frac m \nu}$. 
Then $\phi(\cR)$ is in $\Psi^m$ and its symbol satisfies
$$
\forall a,b,c \in \bN\qquad
\exists k\in \bN, \ C>0\; : \qquad
\|\phi(\pi(\cR))\|_{a,b,c} \leq C \| \phi\|_{\cM_{\frac m \nu} , k}
,
$$
with $k$ and $C$ independent of $\phi$.
\end{proposition}

The proof of Proposition \ref{prop_multipliers} 
can be found in \cite{FR}.
It is based on  Proposition \ref{prop_hula} and the Cotlar-Stein Lemma.

\section{Some properties of symbols, kernels and operators}
\label{sec_prop_symbol_kernel_op}

In this section, we give more explicitly the properties (R1), (R2) and (R4)
given in Introduction.

\subsection{First properties of the symbols}
\label{subsec_1prop_symbols}

The following properties of the symbol $\sigma\in S^m_{\rho,\delta}$ of an
operator with kernel $\kappa_x$
are not difficult to obtain.
\begin{enumerate}
\item If $\beta_o\in \bN_0^n$ then the symbol $X^{\beta_o}_x \sigma(x,\pi)$ is in $S^{m+\delta[\beta_o]}_{\rho,\delta}$ with kernel $X_x^{\beta_o} \kappa_x$
and, 
$$
\|X^{\beta_o}_x \sigma(x,\pi)\|_{S^m_{\rho,\delta}, a,b,c} \leq C_{b,\beta_o}
\| \sigma(x,\pi)\|_{S^m_{\rho,\delta}, a,b + [\beta_o],c}
.
$$
\item If $\alpha_o\in \bN_0^n$ then the symbol $\Delta^{\alpha_o} \sigma(x,\pi)$ is in $S^{m-\rho[\alpha_o]} _{\rho,\delta}$ with kernel $\tilde q_{\alpha_o} \kappa_x$
and, 
$$
\|\Delta^{\alpha_o}\sigma(x,\pi)\|_{S^m_{\rho,\delta}, a,b,c} \leq C_{b,\beta_o}
\| \sigma(x,\pi)\|_{S^m_{\rho,\delta}, a+[\alpha_o],b,c}
.
$$
\item The symbol $\sigma(x,\pi)^*$ is in $S^m_{\rho,\delta}$ with kernel
$\kappa_x^*:y\mapsto \bar\kappa_x(y^{-1})$ and,
$$
\|\sigma(x,\pi)^*\|_{S^m_{\rho,\delta},a,b,c}= 
\!\!\!\!\!\! \sup_{\substack{|\gamma|\leq c \\ [\alpha]\leq a,\, [\beta]\leq b}}\!\!\!\!\!\!
\|
\pi(\id+\cR)^{-\frac{\gamma}\nu }
X_x^\beta\Delta^\alpha \sigma(x,\pi) 
\pi(\id+\cR)^{\frac{\rho [\alpha]-m -\delta[\beta] +\gamma}\nu }\|_{op} 
\, .
$$
\item Let
$\sigma_1\in S^{m_1}_{\rho,\delta}$ and $\sigma_2\in S^{m_2}_{\rho,\delta}$ with 
respective kernels $\kappa_{1x}$ and $\kappa_{2x}$.
Then $\sigma(x,\pi):=\sigma_1(x,\pi)\sigma_2(x,\pi)$ defines the symbol
$\sigma$ in $S^{m}_{\rho,\delta}$, $m=m_1+m_2$,
with kernel
$\kappa_{2x}*\kappa_{1x}$; furthermore
$$
\|\sigma(x,\pi) \|_{S^m_{\rho,\delta},a,b,c}
\leq C
\|\sigma_1(x,\pi) \|_{S^{m_1}_{\rho,\delta},a,b,c+\rho a+|m_2|+\delta b}
\|\sigma_2(x,\pi) \|_{S^{m_2}_{\rho,\delta},a,b,c}
.
$$
where the constant $C=C_{a,b,c} >0$ does not depend on $\sigma$.

Indeed from the Leibniz rule for $\Delta^\alpha$ and $X^\beta$,
the operator
$$
\pi(\id+\cR)^{\frac{[\alpha]-m+\gamma}\nu} X^\beta_x \Delta^\alpha \sigma(x,\pi)\pi(\id+\cR)^{-\frac \gamma \nu}
,
$$
is a linear combination over $\beta_1,\beta_2,\alpha_1,\alpha_2\in \bN_0^n$ satisfying $[\beta_1]+[\beta_2]=[\beta]$, $[\alpha_1]+[\alpha_2]=[\alpha]$, of terms
$$
\pi(\id+\cR)^{\frac{\rho[\alpha]-m -\delta [\beta]+\gamma}\nu} 
X^{\beta_1}_x \Delta^{\alpha_1} \sigma_1(x,\pi)
X^{\beta_2}_x \Delta^{\alpha_2} \sigma_2(x,\pi)
\pi(\id+\cR)^{-\frac \gamma \nu}
,
$$
whose operator norm is less than
\begin{eqnarray*}
\|\pi(\id+\cR)^{\frac{\rho[\alpha]-m-\delta [\beta]+\gamma}\nu} 
X^{\beta_1}_x \Delta^{\alpha_1} \sigma_1(x,\pi)
\pi(\id+\cR)^{-\frac{\rho[\alpha_2]-m_2-\delta[\beta_2]+\gamma}\nu}
\|_{op}\\
\|\pi(\id+\cR)^{\frac{\rho[\alpha_2]-m_2 -\delta[\beta_2] +\gamma}\nu}
X^{\beta_2}_x \Delta^{\alpha_2} \sigma_2(x,\pi)
\pi(\id+\cR)^{-\frac \gamma \nu}\|_{op}.
\end{eqnarray*}

Consequently,  the collection of symbols $\cup_{m\in \bR}S^m_{\rho,\delta}$ forms an algebra.
\item
Using the previous point and the left calculus 
(see Lemma \ref{lem_Xbeta_PsiDO}),
if $\sigma\in S^m_{\rho,\delta}$ with kernel $\kappa_x$, 
then $\pi(X)^{\beta} \sigma \, \pi(X)^{\tilde \beta}$ is in $S^{m+[\beta]+[\tilde \beta]}$ with kernel $X^\beta_y\tilde X^{\tilde \beta}_y \kappa_x(y)$.
\end{enumerate}

\subsection{First properties of the kernels}
\label{subsec_1prop_kernels}

As expected from pseudo-differential calculi on manifolds 
such as homogeneous Lie groups,
the kernels of the operators of order 0 are of Calderon-Zygmund type 
in the sense of Coifman-Weiss \cite[ch. III]{coifman+weiss-LNM71}.
This claim is a consequence of the following proposition together with the properties of the symbols.
\begin{proposition}
\label{prop_kernel}
Assume $\rho\in (0,1]$ and let us fix a homogeneous norm $|\cdot|$ on $G$. 
Let $\sigma \in S^m_{\rho,\delta}$ and $\kappa_x$ the associated kernel.

Then for each $x\in G$, the distribution $\kappa_x$ coincides with a smooth function in $G\backslash \{0\}$. 
Furthermore, $(x,y)\mapsto \kappa(x,y)$ is a smooth function on $G\times (G\backslash \{0\})$, and we have:
\begin{enumerate}
\item 
There exists $C>0$ and  $a,b,c\in \bN_0$ such  that 
for any $y\in G$ with $|y|<1$,
$$
 |\kappa_x(y)| \leq C 
\sup_{\pi\in \Gh} \|\sigma(x,\pi)\|_{S^m_{\rho,\delta},a,b,c}
\left\{\begin{array}{ll}
|y|^{-\frac{Q+m} \rho} & \mbox{if}\ Q+m>0\\
1&\mbox{if}\ Q+m<0\\
\ln |y|&\mbox{if}\ Q+m=0\\
\end{array}\right.
.
$$
\item
For any $M\in \bN_0$,
there exists $C>0$ and  $a,b,c\in \bN_0$ such  that 
for any $y\in G$ with $|y|\geq 1$, 
$$
 |\kappa_x(y)| \leq C 
\sup_{\pi\in \Gh} \|\sigma(x,\pi)\|_{S^m_{\rho,\delta},a,b,c} |y|^{-M}
.
$$
\end{enumerate}
\end{proposition}

For example our operators of order 0 have singularities of the type $|y|^{-Q}$
with $Q$ homogeneous dimension strictly greater than the topological dimension.
Hence the calculus developed here can not coincide with the H\"ormander calculus on $\bR^n$ (abelian).
This contrasts with the compact case: it was shown in \cite{RTW}
that the calculus developed in  \cite{ruzh+turunen_bk2010} on compact Lie groups
leads to the usual H\"ormander operator classes on $\bR^{n}$ extended 
to compact connected manifolds.

\medskip

Before discussing the proof of Proposition \ref{prop_kernel},
let us prove the following couple of easy lemmata.

\begin{lemma}
\label{lem_square_integrable_kernel}
If  $\sigma \in S^m_{\rho,\delta}$ and $a\in \bR$ with $m+a<-Q/2$,
then the distribution $(\id+\cR)^{\frac a\nu} \kappa_x$ 
coincides with a square integrable function for each fixed $x$;
Furthermore, there exists a constant $C=C_{a,m}>0$, 
independent of $\sigma$, such that we have
$$
\forall x\in G\qquad
\|(\id+\cR)^{\frac a\nu}\kappa_x\|_{L^2(G)}\leq C 
\sup_{\pi\in\Gh} \|\sigma(x,\pi)\|_{S^m_{\rho,\delta},0,0,0}
.
$$
\end{lemma}
\begin{proof}[Proof of Lemma \ref{lem_square_integrable_kernel}]
By the Plancherel formula and properties of Hilbert-Schmidt operators,
\begin{eqnarray*}
&&\|(\id+\cR)^{\frac a\nu}\kappa_x\|_{L^2(G)}^2 
= \int_{\Gh} \| (\id+\cR)^{\frac a\nu} \sigma(x,\pi) \|_{HS}^2 d\mu(\pi)
\\ 
&&\qquad\leq \sup_{\pi\in \Gh} \| \pi(\id+\cR)^{\frac{-m}\nu} \sigma(x,\pi)\|_{op}^2
 \int_{\Gh} \|  \pi(\id+\cR)^{\frac{m +a}\nu}\|_{HS}^2 d\mu(\pi)
\\ 
&&\qquad\leq 
\sup_{\pi\in \Gh} \|\sigma(x,\pi)\|_{S^m_{\rho,\delta}, 0,0,0}^2 \ 
\| \cB_{-(m+a)}\|_{L^2(G)} ^2
,
\end{eqnarray*}
which is finite by the properties of Bessel potentials (see Subsection 
\ref{SEC:RO}).
\end{proof}

\begin{lemma}
\label{lem_kernel_continuous}
If  $\sigma \in S^m_{\rho,\delta}$ with $m<-Q$,
then  the associated kernel $\kappa_x(y)$
coincides with a continuous function in $y$ for each $x\in G$.
Furthermore, there exists a constant $C=C_{a,m}>0$, independent of $\sigma$,
such that we have
$$
\forall x,y\in G\qquad
|\kappa_x(y)| \leq C 
\sup_{\pi\in\Gh} \|\sigma(x,\pi)\|_{S^m_{\rho,\delta},0,0,0}
,
$$
\end{lemma}
\begin{proof}
By Lemma \ref{lem_sob_ineq}, 
$$
\|\kappa_x\|_{L^\infty(G)}
\leq C_a
\|(\id+\cR)^{\frac a \nu} \kappa_x\|_{L^2(G)}
\quad\mbox{where}\quad
a= \frac {Q+m}2<-\frac Q2
,
$$ 
and we conclude by 
Lemma \ref{lem_square_integrable_kernel}.
\end{proof}

Consequently, 
for any $\sigma\in S^m_{\rho,\delta}$ with kernel $\kappa_x$,
we can apply Lemma \ref{lem_kernel_continuous}
to the symbol
$$
X^{\beta_o}_x \Delta^{\alpha_o} \pi(X)^\beta \sigma \pi(X)^{\tilde \beta} 
\in S^{m+[\beta]+[\tilde\beta] +\delta[\beta_o]-\rho[\alpha_o]}_{\rho,\delta} 
.
$$
The  kernel is given by
$\tilde X^{\tilde \beta} X^\beta X^{\beta_o}_x \tilde q_{\alpha_o}\kappa_x$  (see  Subsection \ref{subsec_1prop_symbols})
and, if $m+[\beta]+[\tilde\beta] +\delta[\beta_o]-\rho[\alpha_o] <- Q$,
it coincides with a continuous and bounded function,
$$
\forall x,y\in G\quad
|\tilde X^{\tilde \beta}_y X^\beta_y X^{\beta_o}_x \tilde q_{\alpha_o}\kappa_x(y)|
\leq
C \sup_{\pi\in \Gh}\|\sigma(x,\pi)\|_{S^m_{\rho,\delta}, a,b,c}
$$
with $a=[\alpha_o]$,
$b=[\beta_o]$,
$c=\rho [\alpha_o]+ \delta[\beta_o]+ [\tilde \beta]$.

Hence if $\rho>0$ then the kernel of a symbol in $S^m_{\rho,\delta}$
is smooth on $G\times (G\backslash \{0\})$.

Let $|\cdot|_{\nu_o}$ be the homogeneous norm 
defined in 
\eqref{eq_norm_||nuo}
for $\nu_o$ the smallest common multiple of the $\upsilon_j$.
Clearly, for any $p\in \bN$,
 $|\cdot|_{\nu_o}^{2\nu_op}$ is a homogeneous polynomial of degree $p2\nu_o$
and thus can be written as a linear combination of $q_\alpha$ with $[\alpha]=2\nu_op$
(see Lemma \ref{lem_property_qalpha}).
Thus
\begin{eqnarray}
|y|_{\nu_o}^{2\nu_op} |\kappa_x(y)|
&=&
|y^{-1}|_{\nu_o}^{2\nu_op} |\kappa_x(y)|
\nonumber
\\
&\leq &C \sup_{[\alpha ]=2\nu_op}
\sup_{z\in G} |q_\alpha(z) |_{\nu_o}^{2\nu_o} |\kappa_x(g)|
\nonumber
\\
&\leq&
C \sup_{\pi\in \Gh}\|\sigma(x,\pi)\|_{S^m_{\rho,\delta}, a,b,c}
\label{weak_control_kernel}
\end{eqnarray}
as long as $m-\rho 2\nu_o p>Q$. 
Here $C=C_{m,p}$ is independent of $\sigma$.

The estimate \eqref{weak_control_kernel} proves the second point in Proposition \ref{prop_kernel} and is a weaker version of the first.
We will not show this first point 
because of its argument's length. 
The complete proof 
can be found in  \cite{FR}.

\subsection{A pseudo-differential operator as a limit of `nice' operators}

The definition of symbols presented above leads to kernels $\kappa_x$
in the distributional sense
and it is often needed to assume that the kernels are `nice' functions.
In this subsection we explain how we proceed to do so.

We fix a non-negative function 
$\chi_o\in \cD(\bR)$ 
supported in $[1/4,4]$  such that $\chi_o\equiv 1$ on $[1/2,2]$.
For any $\epsilon>0$, 
we write $\chi_\epsilon(x)=\chi_o(\epsilon |x|_{\nu_o}^{2\nu_o})$
where $|\cdot|_{\nu_o}$ is the homogeneous norm 
defined in \eqref{eq_norm_||nuo}
for $\nu_o$ the smallest common multiple of the $\upsilon_j$.
Cleary $\chi_\epsilon\in \cD(G)$.

We denote by $|\pi|$ a `norm' on $\Gh$, 
for example the distance between the co-adjoint orbits of $\pi$ and 1.

By definition, the $\cH_\pi$'s, $\pi\in \Gh$, 
form  a field of Hilbert spaces for the Plancherel measure. 
So we can choose a generating sequence of vectors on $\cH_\pi$
 depending measurably on $\pi$. 
We denote by $\pr_\epsilon$ the orthogonal projection on the $\lceil \epsilon^{-1}\rceil$-th first vectors of this sequence.

Let $\sigma\in S^m_{\rho,\delta}$. 
We consider for any $\epsilon\in (0,1)$, the operator
$$
\sigma_\epsilon(x,\pi) :=\chi_\epsilon(x) 1_{|\pi|\leq \epsilon}
\sigma(x,\pi) \pr_\epsilon
.
$$
Clearly $\sigma_\epsilon\in S^m_{\rho,\delta}$
and for any $a,b,c\in \bN_0$ there exists $C=C_{m,a,b,c}>0$ such that,
$$
\|\sigma_\epsilon\|_{S^m_{\rho,\delta},a,b,c}\leq C
\|\sigma\|_{S^m_{\rho,\delta},a,b,c}.
$$

The corresponding kernel is
$$
\kappa_\epsilon(x,y) =\chi_\epsilon (x) \int_{|\pi|\leq \epsilon}
\tr \left( \sigma(x,\pi) \pr_\epsilon \right) d\mu(\pi)
,
$$
which is smooth in $x$ and $y$ and compactly supported in $x$.
From Proposition \ref{prop_kernel}, $\kappa_{\epsilon,x}$ decays rapidly at infinity in $y$ uniformly in $x$. Furthermore,
point-wise for $ x\in G$ and $y\in G\backslash\{0\}$, 
or in the sense of $\cS'(G)$-distribution for each $x\in G$,
we have the convergence
$\kappa_{\epsilon,x}(y) \underset {\epsilon\rightarrow 0}  \longrightarrow 
\kappa_x(y)$.

Let $T_\epsilon =Op(\sigma_\epsilon)$ be the corresponding operators.
For any $f\in \cS(G)$, $T_\epsilon f\in \cD(G)$ and
$$
Tf(x) = \lim_{\epsilon\rightarrow 0}T_\epsilon f(x)
\quad\mbox{where}\quad T=Op(\sigma)
.
$$

\medskip

In the proofs of the rest of the paper, 
we will assume that the kernels of the operators are sufficiently regular 
and compactly supported in $y$
so that all the performed operations, 
e.g. composition of operators, convolution of kernels, 
group Fourier transform of kernels etc... make sense. 
This can be made rigorous via the procedure described above
since we will always obtain controls in $S^m_{\rho,\delta}$-semi-norms. 

\subsection{Composition}
\label{subsec_composition}

We want to prove
\begin{theorem}
\label{thm_composition}
Let $1\geq \rho\geq \delta \geq 0$ with $\rho\not=0$ and $\delta\not=1$.
If $T_1\in \Psi^{m_1}_{\rho,\delta}$ and $T_2\in \Psi^{m_2}_{\rho,\delta}$
then $T_1T_2\in \Psi^{m}_{\rho,\delta}$ with $m=m_1+m_2$.
\end{theorem}

Let us start with some formal considerations.
Denoting $\sigma_j$ and $\kappa_j$ the symbol and kernel of $T_j$ for $j=1,2$,
it is not difficult to compute the following expression for the composition $T=T_1T_2$,
$$
Tf(x) = \int_G\int_G f(z) \kappa_2(y,z^{-1}y)\kappa_1(x,y^{-1} x) dy dz
.
$$
Thus the kernel of $T$ is 
$$
\kappa_x(w) 
=\int_G\kappa_2(xz^{-1},wz^{-1}) \kappa_1(x,z) dz
.
$$
Using the Taylor expansion for $\kappa_2$ in its first variable,
we have
$$
\kappa_2(xz^{-1},\cdot ) \approx \sum_\alpha \tilde q_\alpha (z) X^\alpha_x 
\kappa_{2x} (\cdot)
\quad\mbox{thus}\quad
\kappa_x(w) \approx \sum_\alpha  X^\alpha_x\kappa_{2x} * \tilde q_\alpha\kappa_1(w)
.
$$
Denoting $\sigma$ the group Fourier transform of $\kappa$, 
we have
$$
\sigma(x,\pi):=\pi(\kappa_x)\approx
\sum_\alpha  \Delta^\alpha \sigma_1(x,\pi) \
X^\alpha_x \sigma_2(x,\pi) .
$$
From Subsection \ref{subsec_1prop_symbols}, 
we know that 
$$
\sum_{[\alpha]\leq M} 
\Delta^\alpha \sigma_1(x,\pi) \
X^\alpha_x \sigma_2(x,\pi) \in S^{m-(\rho-\delta)M}_{\rho,\delta}
.
$$
Hence the main problem is to control the remainder coming from the use of the Taylor expansion;
this is the object of the following lemma.
\begin{lemma}
\label{lem_thm_composition}
We keep the notation defined just above and set
$$
\tau_M(x,\pi):=\sigma(x,\pi) -\sum_{[\alpha]\leq M} 
\Delta^\alpha \sigma_1(x,\pi) \
X^\alpha_x \sigma_2(x,\pi)
.
$$
Let $\beta,\tilde\beta,\beta_o,\alpha_o\in \bN_0^n$.
Then there exists $M_o\in \bN_0$ such that for any integer $M>M_o$, 
there exist $C>0$ and computable integers $a_1,b_1,c_1,a_2,b_2,c_2$ 
(independent of $\sigma_1$ and $\sigma_2$) such that we have
$$
\|\pi(X)^{\beta} \left\{ X_x^{\beta_o} \Delta^{\alpha_o} 
\tau_M
\right\} \pi(X)^{\tilde \beta} \|_{op}
\leq
C
\|\sigma_1\|_{S^{m_1}_{\rho,\delta},a_1,b_1,c_1}
\|\sigma_2\|_{S^{m_2}_{\rho,\delta},a_2,b_2,c_2}
.
$$
\end{lemma}

\begin{proof}[Proof of Lemma \ref{lem_thm_composition}]
Proceeding as in the proof of \eqref{formula_leibniz}, we have:
\begin{eqnarray*}
&&
\tilde q_{\alpha_o}(y)\left(\kappa_x -\sum_{[\alpha]\leq M}
(\tilde q_\alpha \kappa_{1,x})*
(X_x^\alpha \kappa_{2,x})\right)
\\&&\quad=
\sum^{--}_{[\alpha_{o1}]+[\alpha_{o2}]=[\alpha_o]}
\int_G 
 \tilde q_{\alpha_{o1}}(z) \kappa_{1,x}(z)
 \ R^{(\tilde q_{\alpha_{o2}}\kappa_{2,x\cdot})(yz^{-1})}_{0,M}(z^{-1})
\ dz.
 \end{eqnarray*}
 where the sign $\overset{--}\sum$ means linear combination,
here over $[\alpha_{o1}]+[\alpha_{o2}]=[\alpha_o]$, 

Consequently,
$\pi(X)^{\beta_1} X_x^{\beta_o} \Delta^{\alpha_o}
\tau_M(x,\pi)
\pi(X)^{\beta_2}$
is the group Fourier transform of the function of $y$ given by
\begin{eqnarray}
&&\tilde X^{\beta_2}_y 
X^{\beta_1}_y 
X^{\beta_o}_x
\!\!\!\!\!\!\!\!\!\!\!\!
\sum^{--}_{[\alpha_{o1}]+[\alpha_{o2}]=[\alpha_o]}
\!\!
\int_G 
 \tilde q_{\alpha_{o1}}(z) \kappa_{1,x}(z) \
 R^{(\tilde q_{\alpha_{o2}}\kappa_{2,x\cdot})(yz^{-1})}_{0,M}(z^{-1})
\ dz\nonumber
\\&&\quad = 
\!\!\!\!\!\!\!\!\!
\sum^{--}_{\substack{
[\alpha_{o1}]+[\alpha_{o2}]=[\alpha_o]\\
[\beta_{o1}]+[\beta_{o2}]=[\beta_o]}}
\int_G 
 \tilde q_{\alpha_{o1}}(z) X^{\beta_{o2}}_x\kappa_{1,x}(z)\times
 \nonumber
 \\&&\qquad\qquad\qquad\times
 R^{\tilde X^{\beta_2}_y 
X^{\beta_1}_y (\tilde q_{\alpha_{o2}}X^{\beta_{o1}}_x\kappa_{2,x\cdot})(yz^{-1})}_{0,M}(z^{-1})
\ dz. 
 \label{eq_pf_lem2_thm_product_1}
 \end{eqnarray}
 After some manipulations  involving integration by parts and Leibniz rules, 
we obtain that 
$\pi(X)^{\beta_1} X_x^{\beta_o} \Delta^{\alpha_o}
\tau_M(x,\pi)
\pi(X)^{\beta_2}$ is a linear combination 
over
$[\alpha_{o1}]+[\alpha_{o2}]=[\alpha_o]$,
$[\beta_{o1}]+[\beta_{o2}]=[\beta_o]$ and
$[\beta_{11}]+[\beta_{12}]=[\beta_1]$ of 
\begin{equation}
\int_G 
 X^{\beta_{12}}_{z_2=z}\tilde q_{\alpha_{o1}}(z_2) X^{\beta_{o2}}_x\kappa_{1,x}(z_2) \pi(z)^*
  R^{(
(X^{\beta_{o1}}_x X_{x'}^{\beta_{11}} \Delta^{\alpha_{o2}}\sigma_2(xx',\pi)
\pi(X)^{\beta_2})}_{x'=0,M-[\beta_{11}]}(z^{-1})
  dz,
 \label{eq_pf_lem2_thm_product_mainexpression1}
\end{equation}
  where we have extended the notation of the Taylor remainder to accept vector valued functions. 
 The adapted statement of Taylor's estimates given in Proposition \ref{prop_Taylor}
 remains valid. 

We consider each integral \eqref{eq_pf_lem2_thm_product_mainexpression1}
and insert 
$\pi(\id+\cR)^{\frac{e_o}\nu}$
 and its inverse
 with the exponent $e_o$ to be determined in terms of $\beta_{1j},\beta_{oj}, \alpha_{o,j}$, $j=1,2$.
We decompose
$-\frac{e_o}\nu =e_1 +e_2$ with $e_1\in \bN_0$.
 Therefore, each integral in \eqref{eq_pf_lem2_thm_product_mainexpression1}
 is equal to
 \begin{eqnarray}
\int_G
 X^{\beta_{12}}_{z_2=z}\tilde q_{\alpha_{o1}}(z_2) X^{\beta_{o2}}_x\kappa_{1,x}(z_2) \pi(z)^*   \pi(\id+\cR)^{e_1}
 \ \pi(\id+\cR)^{e_2} \nonumber\\
  \{R^{(
(\pi(\id+\cR)^{\frac{e_o}\nu} X^{\beta_{o1}}_x X_{x'}^{\beta_{11}} \Delta^{\alpha_{o2}}\sigma_2(xx',\pi)
\pi(X)^{\beta_2})}_{x'=0,M-[\beta_{11}]}\}(z^{-1})dz
.\nonumber
 \end{eqnarray}
 We can always write
$$
 \pi(z)^*  \pi(\id+\cR)^{e_1}
=  \left(   \sum^{--}_{[\beta']=e_1\nu} X_z^{\beta'}\pi(z) \right)^*,
$$
and, therefore, after integrating by parts, 
\eqref{eq_pf_lem2_thm_product_mainexpression1} is equal to
  \begin{eqnarray*}
  \sum^{--}_{[\beta'_1]+[\beta'_2]=[\beta']=e_1\nu}
  \int_G
X^{\beta'_1}_{z_2=z} X^{\beta_{12}}_{z_2=z}\tilde q_{\alpha_{o1}}(z_2) X^{\beta_{o2}}_x\kappa_{1,x}(z_2) \pi(z)^*  
 \pi(\id+\cR)^{e_2} \\
X^{\beta'_2}_{z_1=z}  \{R^{(
(\pi(\id+\cR)^{\frac{e_o}\nu} X^{\beta_{o1}}_x X_{x'}^{\beta_{11}} \Delta^{\alpha_{o2}}\sigma_2(xx',\pi)
\pi(X)^{\beta_2})}_{x'=0,M-[\beta_{11}]}\}(z_1^{-1})dz
.
 \end{eqnarray*}
 
We fix a pseudo-norm $|\cdot|$ on $G$. 
By  Taylor's estimates and easy manipulations,
we obtain:
\begin{eqnarray*}
&&\|\{ X^{\beta'_2}_{z_1=z}  \{R^{(
(\pi(\id+\cR)^{\frac{e_o}\nu} X^{\beta_{o1}}_x X_{x'}^{\beta_{11}} \Delta^{\alpha_{o2}}\sigma_2(xx',\pi)
\pi(X)^{\beta_2})}_{x'=0,M-[\beta_{11}]}\}(z_1^{-1})\|
\\ &&\leq C \sup_{x_1\in G}
\|\sigma_2(x,\pi)\|_{S^{m_2}_{\rho,\delta}, [\alpha_{o2}], b, [\beta_2]}
\!\!\!\!\!\!
\!\!\!\!\!\!
\!\!\!\!\!\!
\sum_{\substack{
[\gamma]>M-[\beta_{11}]-[\beta'_2]\\
|\gamma|<\lceil M-[\beta_{11}]-[\beta'_2]\rfloor +1}}
\sum_{\substack{[\beta_{o1}']\geq [\beta_{o1}]\\ |\beta_{o1}'|\leq |\beta_{o1}|}}
\!\!\!\!\!\!
|z|^{[\gamma]+[\beta'_{o1}]-[\beta_{o1}]},
\end{eqnarray*}
for $b:=
\max_{|\beta'_{o1}|\leq |\beta_o|} [\beta'_{o1}]+
\max_{|\gamma| <M+1}[\gamma]+e_1\nu +[\beta_1]$.
The last inequality is valid 
if we have chosen $e_o,e_1,e_2$ satisfying for all $\beta_{o1}'$ and $\beta'_2$ 
as in the sums above, the condition
$$
e_o\leq \rho[\alpha_{o2}]-m_2-\delta(
[\beta_{o1}']+[\gamma]+[\beta'_2]+[\beta_{11}])
 -[\beta_2],
$$
which is implied, in particular, by
$$
e_o\leq \rho[\alpha_{o2}]-m_2-\delta([\beta_{o1}]+M)
 -[\beta_2].
$$
It is now time to choose:
\begin{itemize}
\item $e_o=e_1=e_2=0$ 
if $\rho[\alpha_{o2}]-m_2-\delta([\beta_{o1}]+M)
 -[\beta_2]\geq 0$,
\item but if $\rho[\alpha_{o2}]-m_2-\delta([\beta_{o1}]+M)
 -[\beta_2]< 0$,
\end{itemize}
 $$
 e_o:=\rho[\alpha_{o2}]-m_2-\delta([\beta_{o1}]+M)
 -[\beta_2],\quad
 e_1=\lceil -\frac{e_o}\nu\rceil\in \bN,\quad
 e_2=-\frac{e_o}\nu -e_1<0.
 $$  
 We can now go back to
\eqref{eq_pf_lem2_thm_product_mainexpression1}
and obtain that
its operator norm is
 \begin{eqnarray*}
\leq C
\sup_{x\in G}
\|\sigma_2(x,\pi)\|_{S^m_{\rho,\delta}, [\alpha_{o2}], b, [\beta_2]}
  \sum_{[\beta'_1]+[\beta'_2]=e_1\nu}
\int_G |X^{\beta_{12}}(\tilde q_{\alpha_{o1}} X^{\beta_{o2}}_x\kappa_{1,x})(z)|
\\
\sum_{\substack{
[\gamma]>M-[\beta_{11}]-[\beta'_2]\\
|\gamma|<\lceil M-[\beta_{11}]-[\beta'_2]\rfloor +1}}
\sum_{\substack{[\beta_{o1}']\geq [\beta_{o1}]\\ |\beta_{o1}'|\leq |\beta_{o1}|}}
\!\!\!\!\!\!
|z|^{[\gamma]+[\beta'_{o1}]-[\beta_{o1}]}
dz
.
 \end{eqnarray*}
Since $X^{\beta_{12}}(\tilde q_{\alpha_{o1}} X^{\beta_{o2}}_x\kappa_{1,x})$
behaves like a Schwartz function away from the origin
and like $\leq C|z|^{-p_o}$
with $p_o$ depending on $m_1,\beta_{12},\alpha_{o1},\beta_{o2}$
 near the origin  (see Proposition \ref{prop_kernel} and Section \ref{subsec_1prop_symbols}), 
we see that all these integrals converge 
when $-p_o+[\gamma]+[\beta'_{o1}]-[\beta_{o1}]>-Q$
 for all indices as above,
and for this to be true it suffices that
$-p_o+M-[\beta_{11}]+e_o-\nu >-Q$.
We can always find $M$ such that this is satisfied 
and in this case the operator norm of 
\eqref{eq_pf_lem2_thm_product_mainexpression1}
is
$$
\leq C
\sup_{x\in G}
\|\sigma_2(x,\pi)\|_{S^{m_2}_{\rho,\delta}, [\alpha_{o2}], b, [\beta_2]}
\|\sigma_1\|_{S^{m_1}{\rho,\delta},a_1,b_1,c_1}
$$
for some (computable) integers $a_1,b_1,c_1$.

We choose $M$ the smallest integer which is a linear combination over $\bN_0$ of the weights $\upsilon_1,\ldots,\upsilon_n$
such that all the operator norm of  \eqref{eq_pf_lem2_thm_product_mainexpression1}
over 
$[\alpha_{o1}]+[\alpha_{o2}]=[\alpha_o]$,
$[\beta_{o1}]+[\beta_{o2}]=[\beta_o]$,
$[\beta_{11}]+[\beta_{12}]=[\beta_1]$ 
are finite as above.
This proves that 
the operator norm of 
$\pi(X)^{\beta_1} X_x^{\beta_o} \Delta^{\alpha_o}
\tau_M(x,\pi)
\pi(X)^{\beta_2}$
is estimated as stated.
\end{proof}

Hence Lemma \ref{lem_thm_composition} yields 
the following more precise version of Theorem \ref{thm_composition}.

\begin{corollary}
\label{cor_thm_composition}
Under the hypotheses of Theorem \ref{thm_composition},
writing $T_1=Op(\sigma_1)$ and $T_2=Op(\sigma_2)$, 
there exists a unique symbol $\sigma\in S^m_{\rho,\delta}$ such that $T_1T_2=Op(\sigma)$.
Furthermore
$$
\sigma - \sum_{[\alpha]\leq M} \Delta^\alpha \sigma_1 \ X_x^\alpha \sigma_2
\ \in S^{m-(\rho-\delta)M}_{\rho,\delta}
,
$$
and the mapping
$$
\left\{\begin{array}{rcl}
S^{m_1}_{\rho,\delta} \times S^{m_1}_{\rho,\delta}
&\longrightarrow& S^{m-(\rho-\delta)M}_{\rho,\delta} \\
(\sigma_1,\sigma_2) 
&\longmapsto& 
\sigma - \sum_{[\alpha]\leq M} \Delta^\alpha \sigma_1 \ X_x^\alpha \sigma_2
\end{array}\right.
,
$$
is continuous.
\end{corollary}

With similar methods, we can prove that $\Psi^m_{\rho,\delta}$ is stable 
by taking the formal adjoint of an operator, that is, 
if $T\in\Psi^m_{\rho,\delta}$ then $T^*$ defined via
$(Tf_1,f_2)_{L^2} = (f_1,T^*f_2)_{L^2}$ is also in 
$\Psi^m_{\rho,\delta}$:
\begin{theorem}
Let $1\geq \rho\geq \delta \geq 0$ with $\rho\not=0$ and $\delta\not=1$.
If $T\in \Psi^m_{\rho,\delta}$, then its formal adjoint $T^*$
is also in $\Psi^m_{\rho,\delta}$.
More precisely, 
writing $T=Op(\sigma)$
there exists a unique symbol $\sigma^{(*)}\in S^m_{\rho,\delta}$ such that $T^*=Op(\sigma^{(*)})$.
Furthermore
$$
\sigma^{(*)} - \sum_{[\alpha]\leq M} 
\Delta^\alpha X_x^\alpha \sigma^* 
\ \in S^{m-(\rho-\delta)M}_{\rho,\delta}
,
$$
and the mapping
$$
\left\{\begin{array}{rcl}
S^{m}_{\rho,\delta} 
&\longrightarrow& S^{m-(\rho-\delta)M}_{\rho,\delta} \\
\sigma
&\longmapsto& 
\sigma^{(*)} - \sum_{[\alpha]\leq M} 
\Delta^\alpha X_x^\alpha \sigma^* 
\end{array}\right.
,
$$
is continuous.
\end{theorem}

Indeed let us perform formal considerations analogous to the ones for the composition.
Let $T=Op(\sigma)\in \Psi^m_{\rho,\delta}$ with kernel $\kappa_x$.
It is not difficult to compute that the kernel of $T^*$ is $\kappa_x^{(*)}$
given by
$$
\kappa_x^{(*)}(y)=\kappa^*_{xy^{-1}}(y)=\bar \kappa_{xy^{-1}}(y^{-1})
.
$$
Using the Taylor expansion for $\kappa^*_x$ in $x$,
we obtain
$$
\kappa_x^{(*)}(y) \approx \sum_\alpha  
\tilde q_\alpha (y) X^\alpha_x\kappa_x^*(y)
.
$$
Denoting $\sigma^{(*)}$ the group Fourier transform of $\kappa^{(*)}$, 
we have
$$
\sigma^{(*)}(x,\pi):=\pi(\kappa_x^{(*)})\approx
\sum_\alpha 
 \Delta^\alpha X^\alpha_x \sigma(x,\pi)^* .
$$
From Subsection \ref{subsec_1prop_symbols}, 
we know that
$$
\sum_{[\alpha]\leq M} 
 \Delta^\alpha X^\alpha_x \sigma(x,\pi)^* 
\in S^{m-(\rho-\delta)M}_{\rho,\delta}
.
$$
Hence the main problem is as above
to control the remainder coming from the use of the Taylor expansion.
The proof proceeds in a similar way and is left to the reader.

Finally, we note that the proof of Theorem \ref{thm_composition}
can be adapted to provide the treatment of the remainder also
for composition of operators on compact Lie groups in
\cite[Theorem 10.7.8]{ruzh+turunen_bk2010}.


\begin{thebibliography}{99}


\bibitem{bahouri+FK+gallagher_bk2012}
Bahouri, H., Fermanian-Kammerer, C. and Gallagher, I.,
{\em Phase-space analysis and pseudodifferential calculus 
on the Heisenberg group},
Ast\'erisque, \textbf{342} (2012).
See also revised version of March 2013 of
Arxiv:0904.4746.


\bibitem{BJR}
 {Benson, Ch., Jenkins, J. and Ratcliff, G.},
 {The spherical transform of a {S}chwartz function on the
              {H}eisenberg group},
 {\em J. Funct. Anal.},
\textbf{154} {(1998)},
 {379--423}.


\bibitem{cggp}
 {Christ, M., Geller, D., G{\l}owacki, P.
              and Polin, L.},
 {Pseudodifferential operators on groups with dilations.}
  {\em Duke Math. J.},
 \textbf{68}  {(1992)}, 31--65.
 
	
\bibitem{coifman+weiss-LNM71}
{Coifman, R. and Weiss, G.,}
{\em Analyse harmonique non-commutative sur certains espaces
              homog\`enes},
 {Lecture Notes in Mathematics \textbf{242}},
 {Springer-Verlag},
{Berlin},
1971.

\bibitem{dixmier_bk1969}
{Dixmier, J.},
 {{$C\sp*$}-algebras},
 {Translated from the French by Francis Jellett,
              North-Holland Mathematical Library, Vol. 15},
 {1977}.
	
\bibitem{dynin}
 {Dynin, A. S.},
 {An algebra of pseudodifferential operators on the {H}eisenberg
              groups. {S}ymbolic calculus},
 {\em Dokl. Akad. Nauk SSSR},
 \textbf{227}   {(1976)}, 792--795.

\bibitem{RF-cras1}
 {V. Fischer and M. Ruzhansky},
 {Lower bounds for operators on graded {L}ie groups},
  {C. R. Math. Acad. Sci. Paris}, \textbf{351},
 {2013}, {1-2}, pp. {13--18}.
	
\bibitem{FR}
{Fischer, V. and Ruzhansky, M.},
{\em {Q}uantization on nilpotent {L}ie groups}, in preparation.

\bibitem{folland+stein-1974}
 {Folland, G. B. and Stein, E. M.},
 {Estimates for the {$\bar \partial _{b}$} complex and
              analysis on the {H}eisenberg group},
 {Comm. Pure Appl. Math.},
 {27},
 {1974},
 {429--522}.


\bibitem{folland-1975}
 {Folland, G. B.},
 {Subelliptic estimates and function spaces on nilpotent {L}ie
              groups},
 {\em Ark. Mat.},
 \textbf{13}  {(1975)},
 {161--207}.
 
\bibitem{folland-1994}
 {Folland, G. B.},
 {Meta-{H}eisenberg groups},
 {Fourier analysis ({O}rono, {ME}, 1992)},
 {Lecture Notes in Pure and Appl. Math.},
 \textbf{157},
 {121--147},
 {Dekker},
 {New York},
 {1994}.
 

\bibitem{folland+stein_bk82}
{Folland, G. B. and Stein, E.},
{\em Hardy spaces on homogeneous groups},
{Mathematical Notes \textbf{28}},
{Princeton University Press},
1982.

\bibitem{geller}
{Geller, D.},
  {Fourier analysis on the {H}eisenberg group. {I}. {S}chwartz
              space},
 {\em J. Funct. Anal.},
\textbf {36},
  {(1980)},
 {205--254}.
		


\bibitem{helffer+nourrigat-79}
{Helffer, B. and Nourrigat, J.},
{Caracterisation des op\'erateurs hypoelliptiques homog\`enes
              invariants \`a gauche sur un groupe de {L}ie nilpotent
              gradu\'e},
{\em Comm. Partial Differential Equations},
\textbf{4}  {(1979)},
 {899--958}.
 	
\bibitem{hula-1984}
 {Hulanicki, A.},
  {A functional calculus for {R}ockland operators on nilpotent
              {L}ie groups},
   {\em Studia Math.}, 
   \textbf{78} (1984),  253--266.

\bibitem{hula+jenkins+ludwig-1985}
 {Hulanicki, A., Jenkins, J. W. and Ludwig, J.},
 {Minimum eigenvalues for positive {R}ockland operators},
 {\em Proc. Amer. Math. Soc.},
\textbf {94} {(1985)}, 718--720.

\bibitem{ponge}
 {Ponge, R.},
  {\em Heisenberg calculus and spectral theory of hypoelliptic
              operators on {H}eisenberg manifolds},
  {Mem. Amer. Math. Soc.},
 \textbf{194} (906),
 {2008}.

\bibitem{rothschild+stein}
 {Rothschild, L.P. and Stein, E. M.},
 {Hypoelliptic differential operators and nilpotent groups},
 {Acta Math.},
 {137},
 {1976},
 {3-4},
 {247--320}.
	

\bibitem{ruzh+turunen_bk2010}
 {Ruzhansky, M. and Turunen, V.},
 {\em Pseudo-differential operators and symmetries: Background analysis and advanced topics},
 {Pseudo-Differential Operators: Theory and Applications}
 \textbf{2},
  {Birkh\"auser Verlag},
 {2010}.
 
\bibitem{RTi}
Ruzhansky, M. and Turunen, V., 
Global quantization of pseudo-differential operators on 
compact Lie groups, SU(2) and 3-sphere, 
{\em Int Math Res Notices IMRN} (2012), 
58 pages, doi: 10.1093/imrn/rns122.
 
 \bibitem{RTW}
Ruzhansky, M., Turunen, V., and Wirth, J., 
H\"ormander class of pseudo-differential operators on compact Lie groups and global hypoellipticity.
ArXiv:1004.4396.
	
\bibitem{Tnma}
 Taylor, M. E.,
 {\it Noncommutative microlocal analysis. I.} 
 Mem. Amer. Math. Soc. {\bf 52} (1984).
 Revised version accessible at http://math.unc.edu/Faculty/met/ncmlms.pdf
 
 \bibitem{Tnha}
 {Taylor, M. E.},
 {\it Noncommutative harmonic analysis},
  {American Mathematical Society},
  {1986}.


\end{thebibliography}
\end{document}